\documentclass[11pt,a4paper]{amsart}
\usepackage[foot]{amsaddr}

\usepackage{geometry}
\usepackage{amsmath,amsfonts,amssymb,amsthm,mathrsfs}
\usepackage{graphicx}
\usepackage{hyperref}
\usepackage{tikz}
\usepackage{dsfont}
\usepackage{enumitem}

\newtheorem{theorem}{Theorem}[section]
\newtheorem{lemma}[theorem]{Lemma}
\newtheorem{claim}[theorem]{Claim}
\newtheorem{proposition}[theorem]{Proposition}
\newtheorem{corollary}[theorem]{Corollary}
\newtheorem{assumption}[theorem]{Assumption}

\theoremstyle{remark}
\newtheorem{remark}[theorem]{Remark}
\newtheorem{definition}[theorem]{Definition}

\numberwithin{equation}{section}

\newcommand{\M}{\mathbb{M}}
\newcommand{\R}{\mathbb{R}}
\newcommand{\Z}{\mathbb{Z}}
\newcommand{\N}{\mathbb{N}}
\def\P{\mathbb{P}} 
\newcommand{\E}{\mathbb{E}}
\newcommand{\F}{\mathcal{F}}
\newcommand{\calS}{\mathcal{S}}

\newcommand{\id}{\mathds{1}}
\newcommand{\var}{\text{\rm{Var}}}
\newcommand{\eps}{\varepsilon}
\newcommand{\cross}{\text{\textup{Cross}}}
\newcommand{\nodcross}{\text{\textup{NodalCross}}}
\newcommand{\nodarm}{\text{\textup{NodalArm}}}
\newcommand{\ann}{\text{\textup{Ann}}}
\newcommand{\arm}{\text{\textup{Arm}}}

\newcommand{\Rev}{\text{\textup{Rev}}}
\newcommand{\Inf}{\text{\textup{Inf}}}

\begin{document}

\title[Percolation of strongly correlated Gaussian fields II. Sharpness]{Percolation of strongly correlated Gaussian fields~II. Sharpness of the phase transition}
\author{Stephen Muirhead}
\email{smui@unimelb.edu.au}
\address{School of Mathematics and Statistics, University of Melbourne}
\begin{abstract}
We establish the sharpness of the phase transition for a wide class of Gaussian percolation models, on $\Z^d$ or $\R^d$, $d \ge 2$, with correlations decaying at least algebraically with exponent $\alpha > 0$, including the discrete Gaussian free field ($d \ge 3, \alpha = d-2$), the discrete Gaussian membrane model ($d \ge 5, \alpha = d - 4$), and many other examples both discrete and continuous. In particular we do not assume positive correlations. This result is new for all strongly correlated models (i.e.\ $\alpha \in (0,d]$) in dimension $d \ge 3$ except the Gaussian free field, for which sharpness was proven in a recent breakthrough by Duminil-Copin, Goswami, Rodriguez and Severo; even then, our proof is simpler and yields new near-critical information on the percolation density.

For planar fields which are continuous and positively-correlated, we establish sharper bounds on the percolation density by exploiting a new `weak mixing'  property for strongly correlated Gaussian fields. As a byproduct we establish the box-crossing property for the nodal set, of independent interest. 

This is the second in a series of two papers studying level-set percolation of strongly correlated Gaussian fields, which can be read independently.\end{abstract}
\thanks{The author was supported by the Australian Research Council (ARC) Discovery Early Career Researcher Award DE200101467. Part of this research was undertaken while the author was visiting the Statistical Laboratory at the University of Cambridge, and we thank the University for its hospitality. The author would like to thank Roland Bauerschmidt, Vivek Dewan, Alexis Pr\'{e}vost, Franco Severo and Hugo Vanneuville for helpful discussions, Franco for pointing out the references \cite{bgm04,bau13}, and an anonymous referee for valuable feedback on an earlier version.}
\keywords{Gaussian fields, percolation, strongly correlated systems}
\date{\today}
\maketitle

\section{Introduction}
Classically, a percolation model is said to have a \textit{sharp phase transition} if there is a critical parameter $\ell_c$ separating a supercritical \textit{percolating phase} from a subcritical phase of \textit{exponential decay of connectivity}. For \textit{strongly correlated} models -- meaning that the correlations are not integrable at infinity --  the subcritical decay of connectivity can be much slower; in this context \textit{sharpness} refers to the decay being of the same rate, up to constants in the exponent, throughout the subcritical phase. To put it slightly differently, sharpness means that all natural definitions of the critical parameter that relate to the subcritical decay of connectivity ultimately coincide.

\smallskip
The proof of the sharp phase transition for Bernoulli percolation, by Menshikov \cite{men86} and Aizenman-Barsky \cite{ab87}, was a landmark result in percolation theory. Both proofs deduce sharpness as a consequence of differential inequalities; for Menshikov this was
\begin{equation}
\label{e:mens1}
 \theta_R'(\ell) \ge c \, \theta_R(\ell) \Big( \frac{ 1}{ \frac{1}{R} \sum_{r = 1}^{R} \theta_r(\ell) }  - 1 \Big) , 
 \end{equation}
where $\theta_R(\ell)$ denotes the probability that the origin is connected to the boundary of the box $\Lambda_R = [-R,R]^d$ under the parameter $\ell$ and $c > 0$ is a constant, whereas for Aizenman-Barsky it was a pair of partial differential inequalities involving an additional `magnetisation' parameter. As well as establishing sharpness, these inequalities yield near-critical information on the percolation density in the form of the \textit{mean-field bound}
\begin{equation}
\label{e:mfb}
\theta(\ell)  := \lim_{R \to \infty} \theta_R(\ell)  \ge c  (\ell - \ell_c)    
\end{equation}
for a constant $c > 0$ and $\ell > \ell_c$ sufficiently close to $\ell_c$. See also \cite{dct16,dcrt19a,van22} for more recent proofs of these results, and \cite{dcrt19b,s21} for extensions to models whose correlations are unbounded but short-range (i.e.\ integrable at infinity).
 
 \smallskip
For strongly correlated models, despite extensive progress understanding the geometric properties of \textit{strongly} subcritical and supercritical phases \cite{drs14,pr15,grs21,cn21,ms22}, it has remained a significant challenge to prove sharpness. An indication of the difficulty is that one cannot hope to establish the inequalities in \cite{men86,ab87} in general, since this would imply exponential decay in the subcritical regime, which does not occur for all strongly correlated models.

\smallskip
In two recent breakthroughs -- for the Poisson-Boolean model with unbounded radius distribution \cite{dcrt20}, and for level-set percolation of the discrete Gaussian free field (GFF) \cite{dgrs20} -- sharpness was established using a novel `proof by contradiction'. More precisely, in these proofs sharpness was first assumed to be false, and then it was shown that the Menshikov-type differential inequality 
\begin{equation}
\label{e:mens2}
 \theta_R'(\ell) \ge \frac{ c \, \theta_R(\ell)}{ \frac{1}{R} \sum_{r = 1}^{R} \theta_r(\ell) } 
 \end{equation} 
holds within a \textit{ficticious regime} of parameters; although \eqref{e:mens2} implies exponential decay of subcritical connectivity, the use of the ficticious regime ensures there is no conflict with the lack of exponential decay in these models. While \eqref{e:mens2} was established directly for the Poisson-Boolean model, for the GFF an \textit{interpolation scheme} was used to compare the model to a finite-range approximation, for which \eqref{e:mens2} is simpler to establish. Given the use of a ficticious regime, this proof does not yield near-critical information comparable to~\eqref{e:mfb}. We also mention  the recent work \cite{dt22} which studies sharpness in the strongly correlated Poisson-Boolean model using a very different approach specific to that model.
  
\smallskip
In this paper we prove sharpness for a wide class of strongly correlated Gaussian percolation models, which includes the GFF but also other models of interest. Without working in a ficticious regime, we derive a new Menshikov-type inequality
\begin{equation}
\label{e:mens3}
 \theta_R'(\ell) \ge \frac{c \, \theta_R(\ell)}{ \big( \frac{1}{R} \sum_{r = 1}^{R} \theta_r(\ell) \big)^\gamma }  ,  
 \end{equation}
where $\gamma \in (0,1]$ is a constant depending on the model; while weaker than \eqref{e:mens2}, this inequality is still powerful enough to deduce sharpness, and unlike \eqref{e:mens2} is not in conflict with a lack of exponential decay.

\smallskip
Our approach is close in spirit to \cite{dcrt20}, and has several advantages over the proof of sharpness for the GFF in \cite{dgrs20}. First, it is technically much simpler and extends immediately to a wide class of Gaussian percolation models, both discrete and continuous. Second, it yields near-critical information on the percolation density comparable to \eqref{e:mfb}. 

\smallskip
On the other hand, the proof in \cite{dgrs20} yielded different, non-comparable, information about the supercritical regime, namely that the percolating component satisfies a strong local uniqueness property; one can view this as establishing `supercritical sharpness', in contrast with the classical `subcritical' notion of sharpness consider in this paper (see also Remark~\ref{r:supercrit}). It remains an interesting open problem to obtain the analogous supercritical sharpness for a more general class of strongly correlated Gaussian models. On this point, we believe the approach in \cite{dgrs20} could be adapted to cover other \textit{discrete} Gaussian fields besides the GFF, although extending the technique to smooth fields appears challenging.

\smallskip
One commonality among \cite{dcrt20, dgrs20} and the present paper is that they all exploit a \textit{finite-range decomposition} of the models, i.e.\ the fact that the models can be represented as a sum/superposition of \textit{independent} finite-range models on different scales. It remains an ongoing challenge to prove sharpness for strongly correlated models without such a decomposition (see Section \ref{s:open} for further discussion).  

\smallskip
This is the second in a series of two papers studying level-set percolation of strongly correlated Gaussian fields, which can be read independently. The first paper \cite{ms22} gave precise estimates on the decay of $\theta_R(\ell)$ in \textit{strongly subcritical} regimes of strongly correlated Gaussian models, but did not consider the question whether the strongly subcritical regime in fact coincides with the entire subcritical regime. One consequence of the results in this paper is to confirm the coincidence of these regimes (see Corollary~\ref{c:main1} and the surrounding discussion), at least for a subset of the models considered in \cite{ms22}.

\subsection{Main result}
Our main result applies to a class of Gaussian fields which we denote by $\F = \cup_{\alpha > 0} \F_\alpha$ (for `finite-range decomposition'). To keep the discussion brief we defer the definition to Section~\ref{s:class} (see Definition \ref{d:deff}); for now we mention that $\F_\alpha$ consists of stationary Gaussian fields on either $\Z^d$ or $\R^d$, $d \ge 2$, which possess a certain finite-range decomposition and whose (not necessarily positive) correlations satisfy a polynomial decay bound with exponent~$\alpha$. Although $f \in \F$ need not have algebraic correlations $K(x) \sim c |x|^{-\alpha}$, if it does then $f \in \cup_{\alpha' \le \alpha} \F_{\alpha'}$, and in many cases $f \in \F_\alpha$. We present examples in Section~\ref{s:ex}.

\smallskip For a Gaussian field $f \in \F$ and closed sets $A, B \subset \R^d$, we denote by $\{A \leftrightarrow B\}$ the event that $\{f \ge 0\}$ contains a path, defined on either the lattice $\Z^d$ or in $\R^d$ depending on the domain of the field, that intersects both $A$ and $B$. We let $\P_\ell$ denote the law of the field $f + \ell$, abbreviating $\P = \P_0$, and we denote by
\[ \ell_c =  \inf \big\{ \ell : \P_\ell[ 0 \leftrightarrow \infty ]  > 0  \big\} \in [-\infty,\infty] \]
the usual critical parameter of the model. With this notation we define, for $\ell \in \R$ and $R \ge 1$, 
\begin{equation}
\label{e:thetar}
\theta_R(\ell) := \P_\ell[ \Lambda_{1/2} \leftrightarrow \partial \Lambda_R] \qquad \text{and} \qquad \theta(\ell) := \P_\ell[ \Lambda_{1/2} \leftrightarrow \infty]  := \lim_{R \to \infty} \theta_R(\ell).
\end{equation}
The use of $\Lambda_{1/2}$ instead of the origin in \eqref{e:thetar} is mainly for technical reasons; see item (4) in the remarks at the end of this subsection.

\smallskip
As a preliminary, we verify the non-triviality of the phase transition for this class of models:

\begin{proposition}[Non-triviality]
\label{p:nontriv}
For a Gaussian field $f \in \F$, $\ell_c \in (-\infty,\infty)$.
\end{proposition}

The finiteness of $\ell_c$ for $f \in \F$ is a consequence of sprinkled bootstrapping arguments pioneered in \cite{rs13,pr15,pt15}. If moreover $f$ is continuous and isotropic, then we further know from \cite{mrv20} that $\ell_c \le 0$, however we will not use this fact (c.f.\ Corollary \ref{c:main2}).

\smallskip Our main result is the following:

\begin{theorem}[Sharpness]
\label{t:main1}
Consider a Gaussian field $f \in \F_\alpha$, $\alpha > 0$. Then:
\begin{enumerate}
\item For every $\gamma \in (0,\alpha) \cap (0,1]$ and $\ell < \ell_c$ there exists a $c_1 > 0$ such that
\[ \theta_R(\ell) \le  e^{-c_1 R^\gamma  } \ , \quad R \ge 1 . \]
\item For every $\ell_0 > \ell_c$ and $\gamma \in (0, \alpha/ (2(d-1))) \cap (0, 1]$ there exists a $c_2 > 0$ such that
\[  \theta(\ell) \ge  c_2 ( \ell - \ell_c)^{1/\gamma}   \ , \quad  \ell \in (\ell_c, \ell_0) . \]
\end{enumerate}
\end{theorem}

\noindent The first item confirms there is a single regime of subcritical behaviour in the following sense:

\begin{corollary}[Equality of critical parameters]
\label{c:main1}
Let $f \in \F$. Then $\ell_c = \ell^\ast_c$, where
\[     \ell^\ast_c = \inf \big\{ \ell : \liminf_{R \to \infty} \P_\ell[ \Lambda_R \leftrightarrow \partial \Lambda_{2R} ]  > 0 \big\}   . \]
\end{corollary}
\begin{proof}
It is clear that $\ell^\ast_c \le \ell_c$ by definition. Suppose $\ell < \ell_c$. Then by the union bound, stationarity, and the first item of Theorem \ref{t:main1}, as $R \to \infty$,
\[     \P_\ell[ \Lambda_R \leftrightarrow \partial \Lambda_{2R} ]  \le  c_1 R^{d-1}  \theta_R(\ell) \le c_1 R^{d-1} \times  e^{-c_2 R^\gamma  }   \to 0, \]
and so $\ell \le \ell^\ast_c$.
\end{proof}

For parameters $\ell  < \ell^\ast_c$ one can use renormalisation arguments to deduce a detailed description of the connectivity. In particular, in the companion paper \cite{ms22} it is shown that, for a class of smooth isotropic fields with correlations decaying algebraically with exponent $\alpha > 0$, which partially overlaps with class $\F_\alpha$, for every $\ell < \ell^\ast_c$, as $R \to \infty$, 
\begin{equation}
\label{e:subcrit}
- \log  \theta_R(\ell)  \asymp R^{\min\{\alpha,1\}} (\log R)^{-\id_{\alpha=1} }  .  
\end{equation}
Corollary \ref{c:main1} then extends the conclusion of \eqref{e:subcrit} to the entire subcritical phase $\ell < \ell_c$. Analogous bounds to \eqref{e:subcrit} were also recently shown for the GFF \cite{grs21}.

\smallskip
Let us make some additional remarks:

\begin{enumerate}
\item The first item of Theorem \ref{t:main1} is new for all fields $f \in \F$ except for (i) the GFF where it was established in \cite{dgrs20}, (ii) certain planar fields, where it follows by combining the results of \cite{mrv20} and sprinkled bootstrapping arguments, and (iii) the `short-range correlated' case $\alpha > d$, where it is proven in \cite{s21} (under different assumptions). The second item was only known in the short-range case $\alpha > d$, where it can be established by combining results in \cite{s21,dm21} (again under different assumptions).

\smallskip
\item A consequence of the first item is that $\ell_c = 0$ for continuous planar fields in $\F$ which are also \textit{positively-correlated} and \textit{symmetric}:

\begin{corollary}[$\ell_c = 0$]
\label{c:main2}
Consider a Gaussian field $f \in \F$ which is continuous and planar. Then $\ell_c \le 0$. If $f$ is also positively-correlated and its law is invariant under axis reflection and rotation by $\pi/2$, then $\ell_c = 0$.
\end{corollary}
\begin{proof}
For $f$ continuous and planar, the parameter $\ell = 0$ is self-dual, meaning that $\{f \ge 0\}$ and $\{f \le 0\}$ have the same law; a consequence is that, for every $R > 0$,  $\P[A_R] + \P[B_R] = 1$ where $A_R$ (resp.\ $B_R$) is the event that  $[0,R]^2 \cap \{f \ge 0\}$ contains a path that crosses $[0, R]^2$ from left-to-right (resp.\ top-to-bottom) (see item (5) of Corollary~\ref{c:cont}). On the other hand, if $\ell_c > 0$ then by the union bound, stationarity, and the first item of Theorem \ref{t:main1}, as $R \to \infty$, 
\[ \max\{ \P[A_R], \P[B_R] \} \le c_1  R  \, \theta_R(\ell)  \le c_1  R \times  e^{-c_2 R^\gamma  }   \to 0 , \]
 a contradiction. Hence $\ell_c \le 0$.

The complementary bound $\ell_c \ge 0$ does not require that $f \in \F$, and is a general consequence of self-duality, positive association (equivalent to positive correlations for Gaussian fields, see Lemma \ref{l:pa}), stationarity, and invariance under axis reflection and rotation by $\pi/2$, see \cite{gkr88, alex96}.
 \end{proof}

It was previously established in \cite{mrv20} that $\ell_c = 0$ for a wide class of continuous planar Gaussian fields, although with a very different proof and applying to a different class of fields. Notably \cite{mrv20} covers many important fields not in $\F$, for instance the planar monochromatic wave (see Section \ref{s:open}). On the other hand, the proof of $\ell_c \le 0$ in \cite{mrv20} required isotropy, whereas we prove this without assuming any symmetry.

\smallskip
\item As discussed above, the second item of Theorem \ref{t:main1} gives near-critical information on the percolation density. Recall that the critical exponent $\beta$ is defined as 
\[   \theta(\ell)  = (\ell - \ell_c)^{\beta + o(1)} \quad \text{as } \ell \downarrow \ell_c   \]
whenever it exists. Although the existence of $\beta$ is wide open for the models we consider (even the GFF), assuming its existence our result establishes the bound
\[   \beta \le \max\{1, 2(d-1)/\alpha\} .\]

It is an ongoing challenge to prove the existence of the exponent $\beta$ and identify its value in percolation models. Even for Bernoulli percolation on $\Z^d$ it is known only that $\beta$ exists and takes its `mean-field' value $\beta = 1$ in high dimensions $d \ge 11$ \cite{hs90,fh17}, and that the bound $\beta \le 1$, assuming its existence, holds in general \cite{cc86} (for $d=2$ it is strongly believed that $\beta = 5/36$ since this exponent has been computed for site percolation on the triangular lattice \cite{sw01}).

For strongly-correlated models, physicists believe that $\beta = \beta_d(\alpha)$ is universal for all models with algebraic decay of correlations with exponent $\alpha > 0$ \cite{ik91,jgrs20}, but there is ongoing controversy as to its value: while there are theoretical predictions that $\beta_d(\alpha)$ does not depend on $\alpha$, and so has the same value as for Bernoulli percolation on $\Z^d$ \cite{ik91}, numerical work has suggested that $\alpha \mapsto \beta_d(\alpha)$ is strictly decreasing in a certain interval (see \cite{jgrs20} and references therein). Other numerical studies have suggested that $\beta_d(\alpha)$ might even exceed $1$ in the case of the $d=3$ GFF~\cite{ml06} (although we view this as unlikely). Recent work \cite{dpr23} on the closely related but integrable model of the \textit{metric graph Gaussian free field} strongly suggests that $\beta_d(d-2) = 1$, $d \ge 3$, which notably would refute the predictions in \cite{ik91} and \cite{ml06}.

Our result establishes rigorously that, assuming its existence, $\beta_d(\alpha) \le 1$ if $\alpha \ge 2(d-1)$, and also $\beta_d(\alpha) < \infty$ for all $f \in \F$.

\smallskip
\item We emphasise that Theorem \ref{t:main1} does \textit{not} require the field to be positively-correlated, and positive associations (i.e.\ the FKG inequality) play no role in the proof. Nevertheless, many natural examples of fields $f \in \F$ \textit{are} positively-correlated, and in that case $\Lambda_{1/2}$ can be replaced by $0$ in the definitions of $\theta_R(\ell)$ and $\theta(\ell)$ in \eqref{e:thetar}. We give examples of non-positively-correlated fields to which Theorem \ref{t:main1} applies in Section~\ref{s:ex}.

\smallskip
\item While we do not consider fields with slower-than-polynomial decay of correlations (i.e.\ the case $\alpha = 0$), we believe that the proof of Theorem \ref{t:main1} can be adapted to establish sharpness (in the sense of Corollary \ref{c:main1}) for certain fields with poly-logarithmic decay of correlations with sufficiently large exponent. In that case, the subcritical decay of connectivity is slower than stretched-exponential, see \cite{ms22} and also Remark \ref{r:alpha0}.
\end{enumerate}

\subsection{Continuous planar fields}
We next present stronger results in the case of continuous planar fields which are positively-correlated. For this we introduce a second class of fields $\calS = \cup_{\alpha > 0} \calS_\alpha$ (for `scale-mixture'), which roughly speaking are fields whose correlations are positive and regularly varying at infinity with index~$\alpha$, and which also have a certain `scale-mixture' decomposition (see Definition~\ref{d:s} for a precise description). Such fields may also be in $\F$, but do not have to be.

\begin{theorem}
\label{t:main2}
Consider a continuous planar Gaussian field $f \in \F_\alpha \cap \calS$, $\alpha > 0$. Then for every $\ell_0 > \ell_c$ and $\gamma \in (0, \alpha) \cap (0, 1]$ there exists a $c > 0$ such that, for every $\ell \in (\ell_c, \ell_0)$, 
\begin{equation}
\label{e:main2}
 \theta(\ell) \ge  c ( \ell - \ell_c)^{1/\gamma} .  
  \end{equation}
\end{theorem}

\begin{remark}
In fact, in the setting of Theorem \ref{t:main2} we know that $\ell_c = 0$; see Corollary~\ref{c:main2}.
\end{remark}

Note that Theorem \ref{t:main2} establishes the mean-field bound \eqref{e:mfb} as soon as $\alpha > 1$ (as opposed to $\alpha > 2(d-1) = 2$ for Theorem \ref{t:main1}). This is particularly interesting as it includes models lying outside the Bernoulli percolation universality class, which physicists predict corresponds to the parameter regime $\alpha > 3/2$ in the planar case \cite{w84}.

\smallskip
The proof of Theorem \ref{t:main2} relies on a new `weak mixing' property for fields $f \in \calS$ which we state and prove in Section \ref{s:wm} (see Proposition \ref{p:wm}). We next present further applications of this property, which we believe to be of independent interest.

\smallskip
For $a,b, > 0$, let $\nodcross(a,b)$ denote the event that $\{f = 0\} \cap ([0,a] \times [0,b])$ contains a path that crosses $[0,a] \times [0,b]$ from left-to-right. For $R \ge 1$, abbreviate $\arm(R) = \{ \Lambda_{1/2} \leftrightarrow \partial \Lambda_R\}$, and similarly let $\nodarm(R)$ denote the event that $\{ f = 0\}$ contains a path that intersects $\Lambda_{1/2}$ and $\partial \Lambda_R$.

\begin{theorem}[Uniform box-crossing estimates]
\label{t:main3}
Consider a continuous planar Gaussian field $f \in  \calS$. Then for every $\rho > 0$ there is a $ c > 0$ such that
\[ c  \le  \liminf_{R \to \infty} \P[ \nodcross(R, \rho R) ]   \le  \limsup_{R \to \infty} \P[ \nodcross(R, \rho R) ]   \le 1-c. \]
Moreover, if $f \in \calS_\alpha$, $\alpha > 0$, then one may take $c = c_1   e^{-c_2 / \alpha^2}$ for $c_1,c_2 > 0$ depending only on~$\rho$.
\end{theorem}

\begin{corollary}
\label{c:main3}
Consider a continuous planar Gaussian field $f \in \calS$. Then there is a $c > 0$ such that, for every $R \ge 1$,
\[  \P[\arm(R)]  \ge c R^{-1/2} \quad \text{and} \quad  \P[ \nodarm(R)] \ge cR^{-1}    . \]
\end{corollary}

These results were previously only known in the `short-range' case $\alpha > 2$ \cite{bg17, rv19, mv20}. The proof of Theorem \ref{t:main3} relies on recently established uniform box-crossing estimates for general positively-correlated percolation models \cite{kt20}, which apply to the \textit{excursion} sets $\{f \ge 0\}$ and $\{f \le 0\}$. The weak mixing property then allows us to pass these estimates to the \textit{nodal} set $\{f = 0\}$.

\begin{remark}
The box-crossing estimates are a strong indication that the nodal set $\{f = 0 \}$ has a scaling limit, although to establish this rigorously (even along a subsequence) requires additional control on arm exponents \cite{ab99}. In fact, considerations of universality suggest that $\P[ \nodcross(R, \rho R) ]$ converges as $R \to \infty$ to a limit $c_\infty(\rho;\alpha) > 0$ that depends only on $\rho$ and $\alpha$ (and is independent of $\alpha$ for $\alpha > 3/2$). We are unaware of any conjectures in the literature as to the value of  $c_\infty(\rho;\alpha)$, $\alpha < 3/2$, or whether it degenerates as $\alpha \to 0$.
 
\smallskip
Using a similar proof, one can establish other uniform-in-scale properties of the nodal set, for instance uniform bounds on the probability that $\{f=0\}$ contains a circuit in the annulus $\Lambda_{\rho R} \setminus \Lambda_R$, $\rho>1$. For simplicity we restrict our attention to box-crossings.
\end{remark}

\subsection{Examples}
\label{s:ex}
In this section we present examples of fields to which our results apply.

\subsubsection{The Gaussian free field (GFF)}
The GFF is a fundamental example of a strongly correlated system, and has been central to the study of the percolative properties of such systems \cite{blm87,rs13,dpr18}. To define the GFF, let $(X_i)_{i \ge 0}$ be the simple random walk on $\Z^d$, $d \ge 3$, initialised at the origin. The corresponding Green's function
\begin{equation}
\label{e:green}
 G(x) =  \sum_{i \ge 0} \P[X_i = x]   
 \end{equation}
 has asymptotic decay $G(x) \sim c_d |x|^{-(d-2)}$ as $|x| \to \infty$ for a constant $c_d > 0$. The GFF may be defined as the Gaussian field $f$ on $\Z^d$, $d \ge 3$, with covariance 
\[ \E[ f(x) f(y) ] = G(x-y)  \quad \text{for } x,y \in \Z^d. \]
Alternatively, for a finite subset $U \subset \Z^d$ one can consider the field $f = f_U$ with Dirichlet boundary conditions on $U$ (i.e.\ $f_U(x) = 0$ for $x \notin U$) whose density with respect to the Lebesgue measure on $\R^U$ is proportional to
\begin{equation}
\label{e:gff}
 \exp \Big(  - \frac{1}{4d}  \sum_{x \sim y} \big(f(x) - f(y) \big)^2  \Big) .
 \end{equation}
Then the GFF is the weak limit of $f_U$ as $U$ exhausts $\Z^d$.

\smallskip
It is well-known that the GFF admits a finite-range decomposition \cite{bgm04,bau13}. Very recently it has been shown that the GFF is in class $\F_{d-2}$ \cite{sch22}, so Theorem~\ref{t:main1} applies.

\begin{theorem}
\label{t:gff}
The conclusion of Theorem \ref{t:main1} with $\alpha = d-2$ holds for the GFF.
\end{theorem}

As previously mentioned, the first item of Theorem \ref{t:main1} was already known for the GFF \cite{pr15,grs21,dgrs20}, but the near-critical information in the second item is new. This is of particular interest given the disagreement in the physics literature as to the value of the exponent $\beta$. Our result shows that $\beta \le 2(d-1)/(d-2)$ (assuming the exponent exists), and hence that $\beta < 2+\eps$ in sufficiently high dimension.

Note that to establish the sharpness of the phase transition for the GFF (in the sense of Corollary \ref{c:main1}) we require only that the GFF is in class $\cup_{\alpha > 0}\F_{\alpha}$. For completeness we present in Section \ref{s:exfrd} a decomposition due to \cite{dgrs20} that shows the GFF is in class $\F_{(d-2)/2}$.

\subsubsection{The Gaussian membrane model (GMM)}
The GMM is a variant of the GFF in which the discrete gradient in \eqref{e:gff} is replaced by the discrete Laplacian; the effect is to penalise large \textit{curvature} rather than large \textit{gradient} (see \cite{cn21} and references therein for background and motivation). The GMM can also be viewed as the ground state of the `random-field GFF' in which the random field consists of independent Gaussians \cite{dhp21}.

\smallskip
To define the GMM precisely, recall the Green's function $G$ in \eqref{e:green}. The GMM may be defined as the Gaussian field $f$ on $\Z^d$, $d \ge 5$, with covariance
 \[ \E[ f(x) f(y) ] = (G \star G)(x-y)  \quad \text{for } x,y \in \Z^d, \]
where $\star$ denotes convolution; this covariance has asymptotic decay $(G \star G)(x) \sim c_d |x|^{-{d-4}}$ as $|x| \to \infty$ for some $c_d > 0$ \cite[Lemma 5.1]{sak03}. Alternatively, by analogy with \eqref{e:gff}, for a finite subset $U \subset \Z^d$ one can consider the field $f = f_U$ with Dirichlet boundary conditions on $U$ whose density with respect to the Lebesgue measure on $\R^U$ is proportional to
\begin{equation}
\label{e:gmm}
 \exp \Big(  - \frac{1}{2}  \sum_x ( \Delta f(x) )^2  \Big) , 
 \end{equation}
 where $\Delta f(x) =  \frac{1}{2d}  \sum_{y \sim x}  \big( f(y) - f(x) \big)$ is the discrete Laplacian. Then the GMM is the weak limit of $f_U$ as $U$ exhausts $\Z^d$.

\smallskip
Level-set percolation for the GMM was studied in \cite{cn21}, where non-triviality $\ell_c \in (-\infty,\infty)$ was established. However the sharpness of the phase transition was not previously known. Recently it has been shown that the GMM is in $\F_{d-4}$ \cite{sch22}, and so sharpness follows from our main result:

\begin{theorem}
\label{t:gmm}
The conclusion of Theorem \ref{t:main1} with $\alpha = d-4$ holds for the GMM. 
\end{theorem}

For completeness we show in Section \ref{s:exfrd} that the GMM is in $\F_{(d-4)/2}$, which is sufficient to establish the sharpness of the phase transition in the sense of Corollary \ref{c:main1}.

\subsubsection{The Cauchy covariance kernel}
Among smooth strongly correlated Gaussian fields, a central example is the isotropic Gaussian field $F_\alpha$ on $\R^d$, $d \ge 2$, with covariance kernel 
\begin{equation}
\label{e:cauchy}
 \E[F_\alpha(0) F_\alpha(x)] = K_\alpha(|x|) \ , \quad K_\alpha(r) =  \Big(\frac{1}{1 + r^2} \Big)^{\alpha/2} \ , \quad \alpha > 0, 
 \end{equation}
known as the \textit{Cauchy kernel}; see \cite{lo13} and references therein. The strongly correlated case corresponds to $\alpha \in (0,d]$. 

\smallskip The field $F_\alpha$ is real-analytic, and as such cannot possess a finite-range decomposition (see the discussion in Section \ref{s:open}). However, we show that certain \textit{discretisations} of $F_\alpha$ are in class $\F_\alpha$, and so Theorem \ref{t:main1} applies:

\begin{theorem}
\label{t:ck1}
For every $\alpha > 0$ there exists a $c_0 = c_0(\alpha) > 0$ such that for every $c \ge c_0$ the conclusion of Theorem \ref{t:main1} with parameter $\alpha$ holds for the field $f(x) = F_\alpha(c x)|_{\Z^d}$ on $\Z^d$.
\end{theorem}

We also show that $F_\alpha$ is in class $\mathcal{S}_\alpha$, and so the weak mixing property and its consequences in the planar case apply directly to the field $F_\alpha$:

\begin{theorem}
\label{t:ck2}
For every $\alpha > 0$ the planar Gaussian field $F_\alpha$ on $\R^2$ satisfies the conclusions of Theorem \ref{t:main3} with parameter $\alpha$.
\end{theorem}

In addition, all our results apply to certain smooth fields with algebraic decay $\alpha > 0$, which can be viewed as approximations of the field $F_\alpha$: 

\begin{theorem}
\label{t:smooth}
For every $\alpha > 0$ there exists a smooth isotropic Gaussian field $f$ on $\R^d$, $d \ge 2$, whose covariance kernel satisfies
\[  K(x) =  \E[f(0)f(x)]   \sim |x|^{-\alpha} \ , \quad |x| \to \infty, \]
for which the conclusion of Theorem \ref{t:main1} holds with parameter $\alpha$. If $d=2$, then there exists such a field for which the conclusion of Theorems \ref{t:main2} and \ref{t:main3} also hold.
 \end{theorem}

 \subsubsection{Fields with oscillating correlations}

Although the previous examples are all positively-correlated, Theorem \ref{t:main1} does not require positive correlations, not even at infinity. We say that a field has \textit{correlations oscillating on the scale $|x|^{-\alpha}$ if 
\begin{equation}
\label{e:oc}
\liminf_{|x| \to \infty}   K(x)   |x|^\alpha < 0  <  \limsup_{|x| \to \infty}   K(x)   |x|^\alpha   . 
\end{equation}}

\begin{theorem}
\label{t:npc}
For every $\alpha > 0$ there exists a Gaussian field on $\R^d$, $d \ge 2$, with correlations oscillating on the scale $|x|^{-\alpha}$, for which the conclusion of Theorem \ref{t:main1} holds with parameter~$\alpha$.
\end{theorem}

\subsection{Overview of the proof}
\label{s:outline}
To prove the sharpness of the phase transition we use a strategy inspired by \cite{dcrt20} but with some key innovations that exploit the Gaussian setting. The basic idea, introduced in \cite{dcrt19a}, is to consider a randomised algorithm that determines the event $\{\Lambda_{1/2} \leftrightarrow \partial \Lambda_R\}$, and then apply the OSSS inequality to deduce that
\begin{equation}
\label{e:osss1}
 \textrm{Var}[ \id_{\{\Lambda_{1/2} \leftrightarrow \partial \Lambda_R\}} ] = \theta_R(\ell) (1 - \theta_R(\ell) ) \le \frac{1}{2} \sum_i \Rev_i \Inf_i ,   
  \end{equation}
where $\Rev_i$ and $\Inf_i$ denote respectively the revealment probability and the resampling influences of the $i^\textrm{th}$ coordinate in a `factor of i.i.d.' representation of the model (see Section \ref{s:prelim} for precise statements). As shown in \cite{dcrt19a,dcrt20}, for many models one can choose an algorithm satisfying the bound $\Rev_i \le \frac{c}{R} \sum_{r \le R} \theta_r(\ell)$, and then combine \eqref{e:osss1} and a `Russo formula' of the type $\frac{d}{d\ell} \theta_R(\ell) = c \sum_i\Inf_i$ to derive the differential inequality \eqref{e:mens2}.

\smallskip
Recently this strategy was successfully implemented in the case of finite-range Gaussian fields by replacing the Russo formula with a Russo-type inequality $\frac{d}{d\ell} \theta_R(\ell) \ge  c \sum_i \Inf_i$, which is a consequence of Gaussian isoperimetry \cite{dm21}. However, for models with unbounded dependence the strategy is problematic since any algorithm that determines $\{\Lambda_{1/2} \leftrightarrow \partial \Lambda_R\}$ will a priori need to reveal the whole model.

\smallskip
Following \cite{dcrt20}, an elegant way to proceed is to consider an algorithm that reveals coordinates in a \textit{finite-range decomposition} of the model; this way, the algorithm need not reveal all coordinates to determine the event. A key observation is that, even though the revealment probabilities are a priori increasing in the scale $t$ of the decomposition, the Russo-type inequality $\frac{d}{d\ell} \theta_R(\ell) \ge  c t^{\alpha/2} \Inf_i$ \textit{improves at higher scales} (see Proposition \ref{p:russo}). If one uses, as in \cite{dcrt20}, the union bound to control the revealment probabilities at high scales, in the case $\alpha > 2(d-1)$ one recovers \eqref{e:mens2}, and hence sharpness of the phase transition with exponential decay. This gives a new proof of some recent results of Severo \cite{s21}, which used a strategy similar to \cite{dgrs20} (and covered all $\alpha > d$). However, since we are primarily interested in the strongly correlated case $\alpha < d \le 2(d-1)$, we need further ideas.

\smallskip
To extend to smaller $\alpha$ we make another key observation: by monotonicity there exists a scale $t = t_R$ at which the resampling influence of \textit{all coordinates above scale $t$ simultaneously} is approximately equal to $\textrm{Var}[\id_{\{\Lambda_{1/2} \leftrightarrow \partial \Lambda_R\}}]$. At this scale there are two ways to obtain a lower bound for $\frac{d}{d\ell} \theta_R(\ell)$ -- either one applies the OSSS-based method described above to the coordinates \textit{below} this scale, which leads to the inequality 
\begin{equation}
\label{e:bound1}
\theta_R'(\ell) \ge  \frac{ c t^{d-1 - \alpha/2} \theta_R(\ell)}{ \frac{1}{R} \sum_{r = 1}^{R} \theta_r(\ell) } ,  
\end{equation}
 or one applies a Russo-type inequality directly to all coordinates \textit{above} this scale, which leads instead to
   \begin{equation}
\label{e:bound2}
  \theta_R'(\ell) \ge c  t^{\alpha/2}  \theta_R(\ell) . 
  \end{equation}
Even though $t = t_R$ is unknown, taking a \textit{minimum} over $t > 0$ of the \textit{maximum} of \eqref{e:bound1} and \eqref{e:bound2} yields the differential inequality \eqref{e:mens3} (with $\gamma = \alpha/(2(d-1))$), which although weaker than \eqref{e:mens2} is still powerful enough deduce Theorem \ref{t:main1}.

\smallskip
For the stronger bound on the percolation density in Theorem \ref{t:main2}, we use a similar argument but make the additional observation that planarity, continuity, and positive-associations imply that the revealment probabilities $\Rev_i$ at large scales are controlled by the critical one-arm exponent $\eta_1$ (strictly speaking this is only true above criticality, which is sufficient); this leads to the inequality $\beta \le \max\{1, 2\eta_1/\alpha\}$. Since the weak mixing property for strongly correlated fields implies that $\eta_1 \le 1/2$ (see Corollary \ref{c:main3}), this gives Theorem~\ref{t:main2}.

\smallskip
One can observe that this argument establishes the mean-field bound $\beta \le 1$ as soon as we have the inequality $\eta_1 \le \alpha/2$ (see Remark \ref{r:onearm}). In fact we conjecture that $\eta_1 \le \alpha/2$ for all $\alpha > 0$, with equality for sufficiently small $\alpha$.

\subsection{Open questions}
\label{s:open}

In this section we briefly overview some remaining challenges in proving sharpness for Gaussian percolation models.

\subsubsection{Real-analytic fields}

As mentioned above, real-analytic fields -- such as the field $F_\alpha$ with Cauchy covariance kernel \eqref{e:cauchy} -- do not possess a finite-range decomposition, and the sharpness of the phase transition remains open in the strongly correlated case. Indeed, one can observe that a finite-range decomposition would be inconsistent with the field being determined by its values on an open set. Nevertheless, many real-analytic fields enjoy weaker `heat kernel' decompositions of a similar type (see Section \ref{s:exfrd}), in which the covariance kernels in the decomposition have Gaussian, rather than compactly-supported, tails. While it is natural to expect that Gaussian tails can be neglected using approximation arguments, using a naive OSSS-based strategy yields a differential inequality of the form
\[  \theta_R'(\ell) \ge \frac{c(R) \, \theta_R(\ell)}{  \frac{1}{R} \sum_{r = 1}^{R} \theta_r(\ell)   }  ,   \]
for some $c(R) \to 0$ as $R \to \infty$, which is too weak to deduce sharpness.

\smallskip
A possible work-around would be to use an interpolation scheme to compare the model to one with a more suitable decomposition, as in \cite{dgrs20,s21}, however at present this strategy has only been shown to work for smooth fields in the `short-range' case $\alpha > d$.

\smallskip
Another Gaussian field which has a similar `heat kernel' decomposition is the continuum GFF \cite[Section 4.3]{ddg22}; since the continuum heat kernel has unbounded support this decomposition has Gaussian tails. In the GFF literature it is common to work with the smooth mollification of the continuum GFF in which the decomposition is truncated at small scales, and sharpness of the phase transition for this mollified GFF remains open. 

\subsubsection{Monochromatic random waves}
Even more challenging is to prove sharpness for the \textit{monochromatic random waves}, which are the isotropic Gaussian fields on $\R^d$ whose covariance kernel is the Fourier transform of the spherical measure. These fields do not possess any comparable (even approximate) finite-range decomposition, and we currently lack a quantitative understanding of the percolation phase transition even in the planar case (see \cite{mrv20,mui23} for the best known results).

\subsection{Organisation of the paper}
In Section \ref{s:class} we introduce multi-scale decompositions of Gaussian fields, define the classes $\F$ and $\calS$ of fields to which our results apply, and show that the examples in Section \ref{s:ex} belong to these classes. In Section \ref{s:sharp} we establish our main result on the sharpness of the phase transition (Theorem~\ref{t:main1}). In Section~\ref{s:wm} we prove the weak mixing property (Proposition~\ref{p:wm}), and deduce Theorems \ref{t:main2} and \ref{t:main3} as consequences. The appendix collects some technical results on Gaussian fields.

\subsection{Notation}
We use $\N_0$ and $\N$ to denote $\{0, 1, \ldots \}$ and $\{1, 2,\ldots \}$ respectively, and also use $\R^+ = [0,\infty)$. All Gaussian variables and fields are centred unless otherwise specified. A \textit{standard Gaussian} is a Gaussian random variable with unit variance. For a function $q(x,t) : \R^d \times\R^+$ and a multi-index $(k,i) \in \N_0^d \times \N_0$, $\partial^{(k,i)} q $ denotes the derivative $\partial_x^k \partial^i_t q$. A continuous Gaussian field on $\R^d$ is \textit{$C^k$-smooth} if it is in $C^k(D)$ for every compact $D \subset \R^d$. For positive functions $f$ and $g$ we use $f \sim g$ to denote that $f(x)/g(x) \to 1$ as $x \to \infty$.

\smallskip
\section{Multi-scale decompositions of Gaussian fields}
\label{s:class}

Our results apply to a class of Gaussian fields which have a certain \textit{multi-scale white noise decomposition}, which we introduce in this section. We refer to \cite[Chapter 5]{at07} for a general introduction to white noise representations of Gaussian fields.

 Let $W = W(x,t)$ denote the standard white noise on $\R^d \times \R^+$, with the extra dimension being a `scale' parameter. Let $L^2_\text{sym}(\R^d \times \R^+)$ be the subset of $q(x,t) \in L^2(\R^d \times \R^+)$ satisfying $q(x,t) = q(-x,t)$. Then for $q(x,t) \in L_\text{sym}^2( \R^d \times \R^+)$,
\[ f(x) = (q \star_1 W)(x) = \int_{\R^d \times \R^+} q(x-y, t) dW(y,t) \]
defines a stationary Gaussian field on $\R^d$, with covariance kernel
\[ K(x) = \E[f(0)f(x)] =  (q \star_1 q)(x)  = \int_{\R^d \times \R^+} q(x-y, t) q(y,t) \, dy dt, \]
 where $\star_1$ denotes convolution restricted to the first coordinate.
 
 \smallskip We will make the following additional assumptions on $q$:

\begin{assumption}[Smoothness and non-degeneracy]
\label{a:main}
$\,$
\begin{itemize}
\item (Smoothness) For every $t \ge 0$, $q(\cdot,t) \in C^3(\R^d)$. Moreover, for every multi-index $k$ with $|k| \le 3$, $\partial^{k,0} q \in L^2(\R^d \times \R^+)$ and  $\partial^{k,0} q$ is locally bounded on $\R^d \times \R^+$.
\item (Non-degeneracy) There exists a $t \ge 0$ such that $q(\cdot,t)$ is non-zero and $q(\cdot,t') \to q(\cdot, t)$ in $L^1(\R^d)$ as $t' \to t$.
\end{itemize}
\end{assumption} 

\noindent A consequence of the smoothness in Assumption \ref{a:main} (see Lemma \ref{l:cont}) is that $f = q \star_1 W$ is a.s.\ $C^2$-smooth, with 
\begin{equation}
\label{e:wnvar}
\var[ \partial^k f(0) ] = \|\partial^{k,0} q\|^2_{L^2(\R^d \times \R^+)}
\end{equation}
for every multi-index $k$, $|k| \le 2$. Moreover, for every $\ell$, the set $\{f = \ell\}$ is a.s.\ a collection of $C^2$-smooth simple hypersurfaces. We state some more consequences of Assumption \ref{a:main} in Corollary \ref{c:cont}.

\smallskip
We distinguish two, partially overlapping, subclasses of fields $f = q \star_1 W$: 
\begin{enumerate}
\item \textit{Finite-range}, in which $x \mapsto q(x,t)$ has compact support for every $t > 0$.
\item \textit{Scale-mixture}, in which $q$ has the product form $q(x,t) = \sqrt{w(t)} Q(x/t)$. 
\end{enumerate}
We emphasise that, although there exist fields that are in both classes, there are important examples that are in the second class but not the first, e.g.\ the field $F_\alpha$ defined in Section~\ref{s:ex}. 

\smallskip 
For discrete fields on $\Z^d$, it may be more natural to replace the white noise $W$ by a collection of independent standard Gaussians indexed by $\Z^d \times \N$, however this is not without loss of generality (see Remark~\ref{r:spec} below). For our applications it is more convenient to consider a wider class of discrete fields that are restrictions of $q \star_1 W$ to the lattice $\Z^d$. This allows us to treat continuous and discrete fields simultaneously.

\subsection{Finite-range decompositions and the class $\F$}
Recall that a field $f = q \star_1 W$ is said to have a \textit{finite-range decomposition} if  $x \mapsto q(x,t)$ has compact support. By reparametrising~$t$, without loss of generality we may suppose $\text{Supp}(q) \subseteq \{ (x,t) :  |x| \le t/2 \}$, and in that case we say that $q$ has \textit{conic support}. To define the class $\F = \cup_{\alpha > 0} \F_\alpha$ we impose an additional decay assumption on $q$:

\begin{definition}[Class $\F_\alpha$]
\label{d:deff}
A continuous Gaussian field $f$ on $\R^d$, $d \ge 2$, is in class $\F_\alpha$, $\alpha > 0$, if it can be represented as $f = q \star_1 W$ for $q \in L_\text{sym}^2(\R^d \times \R^+)$ satisfying Assumption \ref{a:main} and:
\begin{itemize}
\item(Finite-range) $q$ has conic support.
\item(Decay) There is a $c > 0$ such that, for every $t \ge 1$,
\begin{equation}
\label{e:fdecay}
 \int_{\R^d \times [t,\infty)} q^2(x,s) \, dx ds   \le  c t^{-\alpha} .
 \end{equation}
\end{itemize}
A discrete Gaussian field $f$ on $\Z^d$ is in class $\F_\alpha$, $\alpha > 0$, if it can be represented as $f = (q \star_1 W)|_{\Z^d}$ for $q \in L^2_\text{sym}(\R^d \times \R^+)$ satisfying the same conditions, except perhaps the smoothness in Assumption \ref{a:main}.
\end{definition}

Recall that a Gaussian field on $\R^d$ or $\Z^d$ is said to be $R$-\textit{range dependent} if $\E[f(x)f(y)] = 0$ for all $|x-y| \ge R$, and \textit{finite-range dependent} if it is $R$-range dependent for some $R > 0$. If $f \in \F$, one can decompose $f = \sum_{n \in \N} f_n$ into a sequence $(f_n)_{n \in \N}$ of independent finite-range dependent Gaussian fields as follows. Consider a partition $(B_n)_{n \in \N}$ of $\R^+$ into intervals, and define $f = \sum_{n \in \N} f_{B_n}$, where 
\[ f_{B_n} = (q|_{ \R^d \times B_n} \star_1 W)(x) = \int_{\R^d \times B_n} q(x-y, t) dW(y,t) . \]
By the assumption of conic support,
\[    \text{Supp}( q |_{\R^d \times B_n} )  \subseteq \{(x,t) : |x| \le \sup B_n/2, t \in B_n \} ,  \]
 and so $f_{B_n}(x)$ is measurable with respect to the white noise $W$ restricted to $\{(y,t) :  |y-x| \le \sup B_n/2 , t \in B_n\}$. In particular $(f_{B_n})_{n \in \N}$ is an independent sequence and each $f_{B_n}$ is $(\sup B_n)$-range dependent.

\begin{remark}[Special cases of discrete fields]
\label{r:spec}
Consider a discrete field on $\Z^d$ given by
\[ f(i) =  \sum_{(j,n) \in \Z^d \times \N} \tilde{q}(i-j,n) W(j,n)  , \]
where $\tilde{q}(i, n) \in L^2_\text{sym}(\Z^d \times \N)$ and $ (W(i,n))_{(i,n) \in \Z^d \times \N}$ is a collection of independent standard Gaussians. Then define $q \in L^2_\text{sym}(\R^d \times \R^+)$ as $q(x) = \tilde{q}(i)$ for $x \in (i,n) + (-1/2,1/2]^d \times [0,1)$, $(i,n) \in \Z^d \times \N$, and $q(x) = 0$ otherwise. It is simple to check that $f  \stackrel{d}{=}  (q \star_1 W)|_{\Z^d}$.

\smallskip 
In our applications to the GFF and GMM, it turns out that we require a slightly more general construction involving the `midpoint graph' (see Section \ref{s:exfrd} for why this is necessary). Namely, we define $\M^d$ to be the lattice formed by taking $\Z^d$ and adding vertices at the mid-points of every edge. Then for a function $\tilde{q}_\text{sym}(i, n) \in L^2(\M^d \times \N)$, and $(W(i,n))_{(i,n) \in \M^d \times \N}$ a collection of independent standard Gaussians, we consider the Gaussian field on $\Z^d$
\[ f(i) =  \sum_{(j,n) \in \M^d \times \N} \tilde{q}(i-j,n) W(j,n) . \]
In this case, defining $q \in L^2_\text{sym}(\R^d \times \R^+)$, $q(x)= 2^{-d/2} \tilde{q}(i)$ for $x \in (i,n) + (-1/4,1/4]^d \times [0,1)$, $(i,n) \in \M^d \times \N$, and $q(x) = 0$ otherwise, one can check that $f \stackrel{d}{=} (q \star_1 W)|_{\Z^d}$

\smallskip
We note that the proof of our main results could be slightly simplified in these cases (in particular for the GFF or GMM) due to the absence of certain technicalities in the continuum (see Section \ref{s:cont}), but for brevity we prefer to treat them together with the general case.
\end{remark}

\subsection{Scale-mixture decompositions and the class $\calS$}
Recall that a field $f = q \star_1 W$ is said to have a \textit{scale-mixture decomposition} if $q(x,t) = \sqrt{w(t)} Q(x/t)$; we call $w$ the \textit{weight} and $Q$ the \textit{kernel}. We define the class $\calS$ by imposing additional conditions on $w$ and $Q$; although we only define $\calS$ for continuous fields, it has a natural extension to discrete fields (as for $\F$).

\smallskip
Recall that a non-negative continuous function $\kappa: \R^+ \to \R$ is said to be \textit{regularly decaying (at infinity)} with \textit{index $\alpha > 0$} if it is eventually positive and, for every $v > 0$,
\[ \frac{\kappa(uv)}{\kappa(u)} \to v^{-\alpha} , \quad \text{as } u \to \infty . \]

\begin{definition}[Class $\calS_\alpha$]
\label{d:s}
A continuous Gaussian field $f$ on $\R^d$, $d \ge 2$, is in class $\calS_\alpha$, $\alpha > 0$, if can be represented as $f = q \star_1 W$, $q(x,t) = \sqrt{w(t)} Q(x/t)$, where:
\begin{itemize}
\item  $w \in C^0(\R^+)$ is non-negative, is regularly decaying with index $\alpha+d +1$, and $w(t) t^{-6}$ is bounded.
\item  $Q \in C^3(\R^d)$ is isotropic and satisfies, for all $x \in \R^d$,
\begin{equation}
\label{e:qbounds}
Q(x) \ge  \frac{1}{100} \id_{|x| \le 1/100} \quad \text{and} \quad \max_{|k| \le 3} |\partial^k Q(x)| \le 100 e^{-|x|/100} .
 \end{equation}
\end{itemize}
\end{definition}

\noindent Note that if $q(x,t) = \sqrt{w(t)} Q(x/t)$ satisfies the assumptions in Definition \ref{d:s}, then $q$ also satisfies Assumption \ref{a:main}. Observe also that $f \in \calS$ is isotropic, and we will abuse notation slightly by writing its covariance kernel as $K(R) = \E[f(0)f(x)]$ for $|x| = R \ge 0$.

\begin{remark}
Some of the requirements in $\calS$ are imposed for simplicity, and are not essential for our results. For instance, we could relax isotropy to symmetry under axes reflection and rotation by $\pi/2$, and $Q$ and its derivatives could decay polynomially with large exponent, rather than exponentially. 

\smallskip
Moreover, the constants in \eqref{e:qbounds} are chosen for concreteness, and most of our results are unaffected by their value. The exception is the explicit dependence on $\alpha$ in Theorem \ref{t:main3} (and also \eqref{e:gbound}), where the constants $c_1$ and $c_2$ are only universal once the constants in \eqref{e:qbounds} are fixed. We also note that, if $Q(0) > 0$, then by simply rescaling the pair $(w,Q)$ one can always ensure that the first bound in \eqref{e:qbounds} is satisfied.
\end{remark}

For future use we collect some properties of $f \in \calS$: 

\begin{proposition}
\label{p:smprop}
Consider a Gaussian field $f \in \calS_\alpha$, $\alpha>0$, and let $w$ and $Q$ be respectively the weight and kernel in its scale-mixture representation. Then:
\begin{itemize}
\item As $R \to \infty$, $K(R) \sim c_{\alpha,Q}  w(R) R^{d+1}$, where
\[ c_{\alpha,Q} =  \int_{\R^+}   t^{-\alpha-1}  \times (Q \star Q)(1/t) dt  \in (0, \infty) . \]
\item For every $R >0$, $f$ can be decomposed as
\begin{equation}
\label{e:fra}
f = f_R + g_R , 
\end{equation}
where $f_R$ is a stationary $C^2$-smooth $R$-range dependent Gaussian field, and $g_R$ is a stationary $C^2$-smooth Gaussian field such that, for every multi-index $k$ with $|k| \le 1$,
\begin{equation}
\label{e:gbound}
 \limsup_{R \to \infty}  \frac{ \var[\partial^k g_R(0)] }{w(R) R^{d+1-2|k|} }  \le c_1 c_2^\alpha  \Gamma(\alpha)  
 \end{equation}
for universal constants $c_1,c_2 > 0$, where $\Gamma$ is the Gamma function.
\end{itemize}
\end{proposition}
\begin{proof}
For the first item, by the change of variables  $(y,t) \mapsto (Rty,Rt)$ we have
\begin{align*}
  K(R)  = (q \star_1 q)(R) &=  \int_{\R^d \times \R^+} w(t) Q(|y|/t)Q(|R-y|/t) \, dy dt  \\
  & = w(R) R^{d+1} \int_{\R^+} w(tR)/w(R) \times t^d  \times (Q \star Q)(1/t) \, dt .
  \end{align*}
Since $w(tR) / w(R) \to t^{-\alpha-d-1}$ for every $t > 0$, assuming we may pass to the limit in the integrand yields that $K(R) \sim c_{\alpha,Q} w(R) R^{d+1}$. So let us justify the passage to the limit.

\smallskip
 Fix $\delta \in (0, \alpha)$. By Potter's bounds (see \cite[Theorem 1.5.6]{bgt87}), there exists a $x_0 = x_0(\delta)$ such that, for all $x,y \ge x_0$,
 \begin{equation}
 \label{e:potter}
 0 \le w(y)/w(x) \le 2 \max\{ (y/x)^{-\alpha -d-1- \delta } , (y/x)^{-\alpha - d - 1 +\delta}  \} .
 \end{equation}
 Moreover, by the decay assumption on $Q$ and the Cauchy-Schwarz inequality,
  \begin{equation}
  \label{e:qqbound}
    (Q \star Q )(1/t)   \le  \begin{cases} 
  c_3 e^{-1/(50t)} & t \in (0,1]  , \\
  \|Q\|_{L^2}^2 & t \in [1,\infty) ,
 \end{cases} 
 \end{equation}
 for a constant $c_3 > 0$. Combining these, for $R \ge x_0$ we have
 \[   \id_{ t \ge x_0/R } \times w(tR)/w(R) \times t^d  \times (Q \star Q)(1/t) \le  2 \times  \begin{cases} 
  c_3  e^{-1/(50t)}    & t \in (0,1] , \\
 \|Q\|_{L^2}^2 \times t^{-\alpha - 1 + \delta} & t \in [1,\infty) .
 \end{cases}  \] 
In particular, the right-hand side of the above display is integrable on $\R^+$. Observing also
\begin{align*}
  \int_{t \le x_0/R} w(tR)/w(R) \times t^d  \times (Q \star Q)(1/t) \, dt  & \le  c_4\|w\|_\infty w(R)^{-1}  \int_{t \le x_0/R}  e^{-1/(50t)} \, dt \\
  & \le c_5 w(R)^{-1}  e^{-R/(50 x_0)} \to 0 
   \end{align*}
for $R \ge x_0$ and some $c_4, c_5> 0$, we see that the integral over $t \in (0, x_0/R)$ is negligible. Hence the passage to the limit is justified by the dominated convergence theorem.
 
\smallskip 
For the second item, we fix $\psi: \R^+ \to [0, 1]$ a smooth function such that $\psi(x) = 1$ for $x \le 1/4$ and $\psi(x) = 0$ for $x \ge 1/2$. Then for every $R > 0$, $f$ can be decomposed as $f = f_R + g_R$, where
\[ f_R(x) = \int_{ \R^d \times \R^+} \psi(|x-y|/R) q(x-y,t)  dW(y,t)   \]
and
\[ g_R(x) = \int_{ \R^d \times \R^+} (1-\psi(|x-y|/R)) q(x-y,t) dW(y,t) . \]
By Lemma \ref{l:cont}, both $f_R$ and $g_R$ are stationary $C^2$-smooth Gaussian fields. Since $\psi(|\cdot|/R)q(\cdot,t)$ is supported on $\{|x| \le R/2\}$, the field $f_R$ is $R$-range dependent. To establish the bound \eqref{e:gbound}, by Lemma \ref{l:cont} (specifically \eqref{e:equal}) and the change of variables $(x,t) \mapsto (xtR,tR)$, we have, for $|k| \le 1$,
\begin{align*}
 & \var[\partial^k g_R(0)]    =   \int_{ \R^d \times \R^+} w(t)  \Big( \big(\partial^k \big( ( 1 - \psi(|\cdot|/R) ) Q(|\cdot|/t) \big) \big)(x) \Big)^2  \, dx dt  \\
  &  =  w(R) R^{d+1-2|k|} \\
  & \qquad \times  \int_{\R^+ } w(Rt) / w(R)  \times  t^{d-2|k|} \times \Big(  \int_{\R^d} \Big( \big(\partial^k \big( ( 1 - \psi(t |\cdot|) ) Q(|\cdot|) \big) \big)(x) \Big)^2   dx \Big) dt   .
 \end{align*}
 Since $1-\psi$ is smooth and supported on $\{|x| \ge 1/4\}$, and by \eqref{e:qbounds}, $\var[\partial^k g_R(0)]  $ is at most
  \begin{align*}
 &   c_6 w(R) R^{d+1-2|k|}   \int_{\R^+ } w(Rt) / w(R)  \times  t^{d-2|k|} \times \Big( \int_{ |x| \ge 1/(4t) } \max_{ |k'| \le 1 }   \big(\partial^{k'} Q(|\cdot|)  \big)^2(x)  \, dx \Big) dt \\
    & \quad \le c_7 w(R) R^{d+1-2|k|}   \int_{\R^+ } w(Rt) / w(R)  \times  t^{d-2|k|} \times  \Big( \int_{ |x| \ge 1/(4t) }e^{-|x|/50}  \, dx \Big) dt  \\
    & \quad \le c_8 w(R) R^{d+1-2|k|}   \int_{\R^+ } w(Rt) / w(R)  \times  t^{d-2|k|} e^{-c_9/t} dt  
 \end{align*}
  for constants $c_6, c_7,c_8, c_9> 0$ depending only on $\psi$. Since $w(tR) / w(R) \to t^{-\alpha-d-1}$, and as before justifying the passage to the limit in the integrand using Potter's bounds \eqref{e:potter} and dominated convergence, this yields 
 \[ \frac{ \var[\partial^k g_R(0)] }{w(R) R^{d+1-2|k|} } \le 2 c_8 \int_{\R^+} t^{-\alpha-2|k|-1}  e^{-c_9/t} dt = 2c_8 \, c_9^{\alpha + 2|k|} \, \Gamma(\alpha+2|k|) \le  c_{10} \, c_{11}^{\alpha} \, \Gamma(\alpha)   ,  \]
  for sufficiently large $R$ and universal $c_{10}, c_{11} > 0$. \end{proof}

We also give sufficient conditions for a field with scale-mixture representation to be in $\F$:

\begin{lemma}
\label{l:smf}
Let $f$ be a stationary a.s.\ $C^2$-smooth Gaussian field which can be represented as $f = q \star_1 W$, $q(x,t) = \sqrt{w(t)} Q(x/t)$, where $w$ and $Q$ are non-negative, not identically zero, and satisfy:
\begin{itemize}
\item $w \in C^0(\R^+)$, $w(t) \le c t^{-d-\alpha-1}$ for a $c > 0$ and all $t \ge 1$, and $w(t) t^{-6} $ is bounded.
\item $Q \in C^3(\R^d)$ is supported on $\{|x| \le 1/2\}$, and $Q(x) = Q(-x)$.
\end{itemize}
Then $f \in \F_\alpha$.
\end{lemma}
\begin{proof}
It is clear that $q$ satisfies Assumption \ref{a:main}, and the conic support property is also immediate. To establish \eqref{e:fdecay}, by the change of variables $(x,s) \mapsto (tsx,ts)$ and the bound on~$w$, for all $t \ge 1$,
\begin{align*}
     \int_{\R^d \times [t,\infty)} w(s) Q^2( x /s)\, dx ds  &  =  t^{d+1} \int_{\R^d \times [1,\infty)}  w(ts) s^d \, ds \, \int_{\R^d}  Q^2(x) \, dx   \\
     & \le c_1  t^{-\alpha} \int_{[1,\infty)} s^{-\alpha-1} \, ds \le c_2 t^{-\alpha} 
 \end{align*}
for constants $c_1, c_2> 0$, as required.
\end{proof}
  
As a consequence, we can generate examples of fields $f \in \F_\alpha \cap \calS_\alpha$ for all $\alpha > 0$:
 
\begin{corollary}
 \label{c:sm}
Let $\kappa: \R^+ \to \R$ be regularly decaying with index $\alpha > 0$ such that $\kappa(x) \le c |x|^{-\alpha}$ for a constant $c > 0$ and all $x \ge 1$. Then there exists a smooth isotropic field $f \in \F_\alpha \cap \calS_\alpha$ whose covariance kernel satisfies $\E[f(0) f(x)] \sim \kappa(|x|)$.
\end{corollary}
\begin{proof}
Let $w$ be a smooth non-negative bounded function such that $w(t) \sim  \kappa(t) t^{-d-1}$ as $t \to \infty$ and $w(t) e^{1/t} \to 0$ as $t \to 0$, let $Q$ be a smooth non-negative isotropic function supported on $\{|x| \le 1/2\}$, and define $f = q \star_1 W$, $q(x,t) = \sqrt{w(t)} Q(x/t)$. Then $f$ is smooth (see Remark \ref{r:smooth}), is in $\calS_\alpha$ by definition, is in $\F_\alpha$ by Lemma \ref{l:smf}, and by Proposition~\ref{p:smprop} satisfies $K(R) \sim  c_1 w(R) R^{d+1} \sim c_1 \kappa(|x|)$ for some constant $c_1  > 0$. Since multiplying $w$ by a constant scales $K$ by the same constant, we can fix $c_1= 1$ by scaling~$w$ appropriately.
\end{proof}

\subsection{Examples}
\label{s:exfrd}

We next present examples of fields in $\F$ and $\calS$, including those in Section~\ref{s:ex}. This verifies that Theorems \ref{t:ck1}--\ref{t:npc} are corollaries of our main results.

\subsubsection{The Cauchy covariance kernel}
Recall the smooth isotropic Gaussian field $F_\alpha$ with Cauchy covariance $K_\alpha$ defined in \eqref{e:cauchy}. We consider two scale-mixture decompositions of~$F_\alpha$: 
\begin{enumerate}
\item With kernel $Q(x) = \id_{|x| \le 1}$, i.e.\ the indicator of the ball in~$\R^d$, which verifies that a sufficiently coarse discretisation of $F_\alpha$ is in $\F_\alpha$.
\item With $Q(x) = e^{-|x|^2/2}$, i.e.\ the (scaled) heat-kernel, which shows that $F_\alpha \in \calS_\alpha$.
\end{enumerate} 

\smallskip
\noindent \textit{(1) Decomposition over the indicator of the ball.} Define $\chi_t(x) = \id_{|x| \le t}$. It turns out that one can decompose arbitrary smooth isotropic functions as a scale mixture of $\chi_t \star \chi_t$:

\begin{proposition}[{\cite[Theorem 1, Remark 1]{hs02}}]
\label{p:decomp}
For $d \ge 2$, let $K$ be an isotropic function on $\R^d$ that is $(d+1)$-times differentiable away from $0$, and suppose that, for all $k \le d+1$, $K^{(k)} (s) s^{|k|} \to 0$ as $s \to \infty$, where $K^{(k)}$ denotes the $k^\textrm{th}$ derivative in the first coordinate direction. Then
\[ K(x) =  \int_0^\infty w(t) (\chi_{t/2} \star \chi_{t/2})(x) dt , \]
where the weight $w$ is given by
\begin{equation}
\label{e:weight}
 w(t) = c_d  t^{1-d} \int_t^\infty  (-1)^{d+1}  K^{(d+1)}(s) s^{d-2} (1 - (t/s)^2)^{(d-3)/2}   \, ds ,
 \end{equation}
for $c_d > 0$ an explicit constant depending only on the dimension.
\end{proposition} 

Now suppose $f$ is a smooth isotropic Gaussian field on $\R^d$ whose covariance kernel satisfies the assumptions of Proposition \ref{p:decomp}. If the associated weight $w$ is non-negative, $f$ has a scale-mixture finite-range decomposition
\[ f \stackrel{d}{=} q \star_1 W \ , \quad q(x,t) = \sqrt{w(t)} \chi_{t/2}(x) . \]
If $f$ is real-analytic, and so lacks a finite-range decomposition, the weight $w$ cannot be non-negative, although it may be \textit{eventually} non-negative. This is the case for the Cauchy kernel:

\begin{proposition}
\label{p:ckdecomp}
For every $d \ge 2$ and $\alpha > 0$, $K_\alpha$ satisfies the assumptions of Proposition~\ref{p:decomp}. Moreover, the associated weight $w = w_\alpha$ is eventually non-negative, and there is a $c > 0$ such that, for all $t \ge 1$,
\begin{equation}
\label{e:ckdecomp}
 w_\alpha(t) \le c t^{-\alpha - d -1}   . 
 \end{equation}
\end{proposition}
\begin{proof}
We first claim that, for all $k \in \N_0$, as $x \to \infty$,
\begin{equation}
\label{e:cauchyregvar}
 K_\alpha^{(k)}(x) \sim (-1)^k \prod_{i=0}^{k-1} (\alpha+i) x^{-\alpha-k} .
 \end{equation}
 To verify this, define functions
 \[  f_1(x) = (1+x)^{-\alpha/2}   \quad \text{and} \quad f_2(x) = x^{-\alpha/2} ,  \]
 and observe that, for all $k \in \N_0$, as $x \to \infty$, $ \partial^k f_1(x)  \sim \partial^k f_2(x)$. Then applying the chain rule for higher derivatives (i.e.\ Fa\`{a} di Bruno's formula) to $f_i$ and $g(x) = x^2$, we see that, as $x \to \infty$, 
 \[ K_\alpha^{(k)}(x)  = \partial^k f_1(x^2)   \sim \partial^k f_2(x^2) = \partial^k x^{-\alpha}, \]
  from which \eqref{e:cauchyregvar} follows.

In light of \eqref{e:cauchyregvar}, $K_\alpha$ satisfies the assumptions of Proposition \ref{p:decomp}. Moreover, there exists a $t_0 > 0$ such that $(-1)^{k}K_\alpha^{(k)}(t) \ge 0$ for all $k \le d+1$ and $t \ge t_0$. Considering \eqref{e:weight}, this shows that $w_\alpha(t) \ge 0$ for $t \ge t_0$. To prove \eqref{e:ckdecomp} we separate into the cases $d \ge 3$ and $d =2$. In the former case, since $(1 - (t/s)^2)^{(d-3)/2}  \in [0,1]$ for all $s \ge t \ge 0$, integrating \eqref{e:weight} by parts gives that
\begin{equation}
\label{e:walpha1}
   w_\alpha(t) \le c_d t^{1-d} \int_t^\infty  (-1)^{d+1}  K_\alpha^{(d+1)}(s) s^{d-2}   \, ds   \le  c_d  (-1)^d   t^{-1}    K_\alpha^{(d)}(t)  
   \end{equation}
   for all $t \ge t_0$. In the case $d=2$ we split the integral in \eqref{e:weight} at $2t$. Then since $(-1) K_\alpha^{(3)}(s)$ is non-increasing on $s \ge t \ge t_0$, and using the change of variables $s \mapsto ts$, we have
\begin{equation}
\label{e:walpha2}
 w_\alpha(t)   \le  c_d  t^{-1} \Big(    -   K^{(3)}(t)  t \int_1^{2}  (1 - s^{-2})^{-1/2}   \, ds +   (3/4)^{-1/2}   K^{(2)}(2t)   \, ds \Big)  
    \end{equation}
for all $t \ge t_0$. Combining the bounds \eqref{e:walpha1}--\eqref{e:walpha2} with \eqref{e:cauchyregvar} gives \eqref{e:ckdecomp}.
\end{proof}

\smallskip
\noindent As a consequence of Proposition~\ref{p:ckdecomp}, we show that a sufficiently coarse discretisation of $F_\alpha$ is in class~$\F_\alpha$:

\begin{proposition}
\label{p:cauchyf}
For every $\alpha > 0$ there exists a $c_0 = c_0(\alpha) > 0$ such that for every $c \ge c_0$ the discrete Gaussian field $F_\alpha(c x)|_{\Z^d}$ is in class $\F_\alpha$.
\end{proposition}
\begin{proof}
Recall the decomposition 
\[ K_\alpha(x) =  \int_0^\infty w_\alpha(t) (\chi_{t/2} \star \chi_{t/2})(x) dt , \]
guaranteed to exist by Propositions \ref{p:decomp} and \ref{p:ckdecomp}. By the conclusion of Proposition \ref{p:ckdecomp}, and since 
\[  \int_0^\infty w_\alpha(t) \| \chi_{t/2} \|_{L^2}^2  \, dt = K_\alpha(0) > 0    , \]
we may choose $t_0 = t_0(\alpha) > 0$ and $c_1 > 0$ such that, for all $t \ge t_0$,
\[  0 \le w_\alpha(t)  t^{\alpha + d + 1} \le c_1 \quad \text{and} \quad \int_0^t w_\alpha(s) \| \chi_{s/2} \|_{L^2}^2  ds  \ge 0 .\]
Now fix $c \ge c_0 := 2t_0$. We claim that $F_\alpha(cx)|_{\Z^d}$ can be represented as $\tilde{f} = (q \star_1 Z)|_{\Z^d}$, where
\begin{align*}
& q(x,t) = \sqrt{c} \times \chi_{ct/2}(x) \\
& \qquad \qquad \qquad \times \begin{cases}
 \Big(  \int_0^{1/2} \| \chi_{cs/2} \|_{L^2}^2 ds \Big)^{-1/2} \Big(  \int_0^{1/2} w_\alpha(cs) \| \chi_{cs/2} \|_{L^2}^2  ds \Big)^{1/2}   &  t < 1/2 , \\
 \sqrt{w_\alpha(ct)}  & t \ge 1/2 .
\end{cases} \end{align*}
Indeed $ \var[ \tilde{f}(0) ] $ is equal to
\[  \int_{\R^d \times \R^+} \! \! q^2(x,t) \, dx dt = c \int_0^\infty w_\alpha(cs) \| \chi_{cs/2}\|_{L^2}^2 \, ds    =   \int_0^\infty w_\alpha(s) \| \chi_{s/2}\|_{L^2}^2 \, ds  = \var[F_\alpha(0)] . \]
 Moreover, for every $i \in \Z^d \setminus \{0\}$, by the change of variables $t \mapsto t/c$,
 \begin{align*}
 \textrm{Cov}[ \tilde{f}(0)\tilde{f}(i) ] &  = c \int_{1/2}^\infty w_\alpha(ct) (\chi_{ct/2} \star \chi_{ct/2})(i) =   \int_{c/2}^\infty w_\alpha(t) (\chi_{t/2} \star \chi_{t/2})(i)  \\
 & = \textrm{Cov}[ F_\alpha(0) F_\alpha(ci) ] , 
 \end{align*}
 which proves the claim. Finally, observe that $q$ is supported on $\{(x,t) : |x| \le t/2 \}$, and
 \[ \int_{\R^d \times [t,\infty)} q^2(x,s) \, dx ds =  \int_t^\infty  w_\alpha(cs)   \| \chi_{s/2}\|_{L^2}^2 \, ds  \le c_2 t^{-\alpha }  \]
for all $t \ge 1$ and some $c_2 > 0$, where we used that $w_\alpha(cs) \le c_3 s^{-\alpha-d-1}$ and $\|\chi_{s/2}\|_{L^2}^2 \le c_4 s^d $ for some $c_3,c_4 > 0$. We conclude that $\tilde{f} \in \F_\alpha$, completing the proof.
\end{proof}

\begin{proof}[Proof of Theorem \ref{t:ck1}]
Combine Proposition \ref{p:cauchyf} with Theorem \ref{t:main1}.
\end{proof}

\smallskip
\noindent \textit{(2) Heat-kernel decomposition.} It is well-known that $r \mapsto (1+r)^{-\alpha/2}$ is the Laplace transform of the function
\[ v(s) = \Gamma(\alpha/2)^{-1} s^{\alpha/2-1} e^{-s},  \ s > 0 ,\]
and so the Cauchy kernel $K_\alpha$ can be decomposed as a scale-mixture of Gaussians 
\begin{align*}
 K_\alpha(x) & = \int_0^\infty v(s) e^{-s|x|^2} \, ds   =   \int_0^\infty 2^{-1} t^{-3} v(1/(4t^2)) e^{-|x|^2/(4t^2)} \, dt \\
 & = (2 \Gamma(\alpha/2))^{-1} \int_0^\infty t^{-\alpha-1} e^{-1/(4t^2)} e^{-|x|^2/(4t^2)} \, dt , 
 \end{align*}
 where we used the change of variables $s \mapsto 1/t^2$. As a consequence (see \cite[Proposition A.4]{ms22}), $F_\alpha$ can be represented as $q_\alpha \star_1 W$, for
\begin{equation}
\label{e:qalpha}
q_\alpha(x,t) = \sqrt{w_\alpha(t)} Q(|x| / t) , 
\end{equation}
where $w_\alpha(t) =  c_{d,\alpha}  t^{-\alpha-d-1} e^{-1/(4t^2)} $, $Q(x) = e^{-|x|^2/2}$, and $c_{d,\alpha} > 0$ is an explicit constant; this is analogous to the `heat-kernel decomposition' of the continuum GFF \cite[Section 4.3]{ddg22}. 

\smallskip
Using this decomposition it is straightforward to complete the proof of Theorem \ref{t:ck2}:

\begin{proof}[Proof of Theorem \ref{t:ck2}]
Eq.\ \eqref{e:qalpha} shows that $F_\alpha \in \calS_\alpha$, so this follows from Theorem \ref{t:main3}.
\end{proof}

\subsubsection{Proof of Theorem \ref{t:smooth}}

This is immediate from the properties of $\mathcal{F}$ and $\mathcal{S}$ developed the previous section. More precisely, we combine Corollary \ref{c:sm} with Theorems \ref{t:main1} and \ref{t:main2}--\ref{t:main3}.

\subsubsection{Fields with oscillating correlations}

Finally we verify the existence of fields in the class $\mathcal{F}_\alpha$ whose correlations oscillate in the sense of \eqref{e:oc}:
 
\begin{proposition}
\label{p:npc}
For every $\alpha > 0$ there exists a Gaussian field $f \in \mathcal{F}_\alpha$ with correlations oscillating on the scale $|x|^{-\alpha}$.
\end{proposition}

\begin{proof}
We first claim that there exist $\eps \in (0,1)$, $k_1,k_2 \ge 2$, and a smooth function $\tilde{q} : \R^d \to \R$ supported on the ball $\{|x| \le 1/2\}$, such that $\tilde{q}(x) = \tilde{q}(-x)$, and such that $\widetilde{K} = \tilde{q} \star \tilde{q}$ satisfies $\widetilde{K}(e_1/k_1) = - \eps$, $\widetilde{K}(e_1/k_2) = \eps$, and $\| \widetilde{K} \|_\infty = 1$, where $e_1$ is the unit vector in the first basis direction. Indeed one can first consider the piece-wise function
\[     \hat{q}(x) = \id_{[-1,1]^d}(x) - \id_{ 2 e_1 + [-1,1]^d}(x) - \id_{-2 e_1 + [-1,1]^d}(x), \]
for which $(\hat{q} \star \hat{q})(0) = 3 \times 2^d$ and $(\hat{q} \star \hat{q})(2e_1) = - 2 \times 2^d$, and then take $\tilde{q}$ to be a smooth approximation of $\hat{q}$ and apply a linear rescaling.

Now fix $k \ge \max\{k_1,k_2\} \ge 2$ such that $k^\alpha ( 1-  k^{-\alpha} ) > 2/\eps$, define $t_n = k^n$ and $w_n = t_n^{-\alpha}$, and consider the field $f = q \star_1 W$ where 
  \[ q(x,t) = \sum_{n \in \N} \sqrt{w_n} \tilde{q}(x/t_n) \id_{t \in [t_n-1,t_n]} .\]
 We can check immediately that $q$ satisfies Assumption \ref{a:main} and that $q$ has conic support. By construction we also have
 \[  \int_{\R^d \times [t,\infty) }q^2(x,s) \,dx ds \le  \sum_{n \in \N: t_n \ge t} w_n    \quad \text{and} \quad  K(x) = \sum_{n \in \N: t_n \ge x} w_n \widetilde{K}(x/t_n)  ,   \]
 and moreover we have, for each $n \in \N$,
 \[ \sum_{m \ge n+1} w_m = w_n \frac{k^{-\alpha}}{1 - k^{-\alpha}}   < (\eps/2) w_n.\]
 This implies that, for some $c > 0$,
 \[ \int_{\R^d \times [t,\infty) } q^2(x,s) \,dx ds \le c t^{-\alpha} \]
 and so  $f \in \mathcal{F}_\alpha$. Finally, recalling that $ \|\widetilde{K} \|_\infty = 1$, for all $n \in \N$,
 \[   K( t_n e_1 / k_1)  = w_n \widetilde{K}(1/k_1) +    \sum_{m \ge n+1} w_m \widetilde{K}(t_n k_1 / t_m) \le - \eps w_n  +(\eps/2) w_n  =-  (\eps/2) t_n^{-\alpha}, \]
 and similarly $K( t_n e_1 / k_2) \ge   (\eps/2) t_n^{-\alpha}$. Hence $K$ satisfies \eqref{e:oc}.
 \end{proof}

\begin{proof}[Proof of Theorem \ref{t:npc}]
Combine Proposition \ref{p:npc} with Theorem \ref{t:main1}.
\end{proof}

\subsubsection{Heat kernel decompositions of the Gaussian free field and membrane models}

Recall that very recently it has been shown that the GFF and GMM are in class $\mathcal{F}_{d-2}$ and $\mathcal{F}_{d-4}$ respectively \cite{sch22}. For completeness we present a simple argument, following \cite{dgrs20}, that shows the GFF and GMM are in $\mathcal{F}_{(d-2)/2}$ and $\mathcal{F}_{(d-4)/2}$ respectively, which is already sufficient to yield the sharpness of the phase transition in the sense of Corollary \ref{c:main1}. This is based on a heat kernel decomposition of the Green's function.

\smallskip
Recall that $(X_n)_{n \in \N_0}$ denotes the simple random walk (SRW) on $\Z^d$, with $G$ its Green's function. Define the `lazy' SRW $(X^\text{lazy}_n)_{n \in \N_0}$ on $\Z^d$ which stays put with probability $1/2$ and moves to one of its neighbours uniformly at random with probability $1/2$; by decomposing into excursions we have that, for $x \in \Z^d$,
\[ G(x) = \sum_{n \in \N_0} \P(X_n = x) = \frac{1}{2}  \sum_{n \in \N_0} \P(X^{\textrm{lazy}}_n = x) .  \]
Recall also the lattice $\M^d$, formed from adding to $\Z^d$ vertices at midpoints of all edges, and let $(X^{\M^d}_n)_{n \in \N_0}$ denote the simple random walk on $\M^d$; the trace of $X^{\M^d}$ on $\Z^d$ is $X^\text{lazy}$. By the Markov property and stationarity we have, for $x \in \Z^d$ and $n \in \N$,
\begin{align*}
\label{e:lazy}
  \P(X^\text{lazy}_n = x) = \P(X^{\M^d}_{2n} = x) &=  \sum_{y \in \M^d}  \P( X^{\M^d}_n = y ) \P( X^{\M^d}_{2n} = x  \, | \,  X^{\M^d}_n = y  ) \\ & = \sum_{y \in \M^d}  \P( X^{\M^d}_n = y ) \P( X^{\M^d}_n = x-y)  .  
  \end{align*}
Defining $q(x,n) = \frac{1}{\sqrt{2}} \P( X^{\M^d}_{n-1} = x )$, we conclude that $G(x)$ has the representation \begin{equation}
\label{e:gfdecomp}
  G(x) = \sum_{(y,n) \in \M^d \times \N } q(y,n) q(x-y,n) . 
  \end{equation}
Note that the use of the lattice $\M^d$ was crucial in this representation, since it allows the self-convolution representation of the kernel $\P(X^\text{lazy}_n = x)$.
 
\smallskip We next derive an analogous decomposition for the convolution $(G \star G)(x)$. Let $Y^\text{lazy}$ denote an independent copy of $X^\text{lazy}$. Then by the Markov property and stationarity 
\begin{align*}
 (G \star G)(x) =  \sum_{y \in \Z^d} G(x-y) G(y) &   = \frac{1}{4} \sum_{n,m \in \N_0}  \sum_{y \in \Z^d}  \P(X^\text{lazy}_n = x-y, Y^\text{lazy}_m = y ) \\
 & = \frac{1}{4}  \sum_{n,m \in \N_0} \P(X^\text{lazy}_{n+m}= x)  = \frac{1}{4}  \sum_{n \in \N_0} (n+1) \P(X^\text{lazy}_n = x) .
 \end{align*}
Defining $\tilde{q}(x,n) = \frac{\sqrt{n}}{2} \P( X^{\M^d}_{n-1} = x )$, we see that
\begin{equation}
\label{e:gfdecomp2}
( G \star G)(x)  = \sum_{(y,n) \in \M^d \times \N } \tilde{q}(y,n) \tilde{q}(x-y,n) . 
  \end{equation}

\begin{lemma}
\label{l:gffex}
 The GFF and GMM are in class $\F_{(d-2)/2}$ and $\F_{(d-4)/2}$ respectively.
 \end{lemma}
\begin{proof}
Recall that $(W(x,n))_{(x,n) \in \M^d \times \N}$ is a collection of independent standard Gaussians. The decomposition \eqref{e:gfdecomp} implies that the stationary Gaussian field
\[  f(x) =  \sum_{(y,n) \in \M^d \times \N } q(x-y,n-1) W(y, n)  \]
has covariance kernel $G(x)$, and is hence distributed as the GFF. Furthermore, the support of $q$ is contained in $\{ (x,n) : |x|_\infty \le  (n-1)/2\}$, and by the local central limit theorem $q(x,n)$ is bounded by $c_1 n^{-d/2}$ with Gaussian tails on scale $n^{-1/2}$, which implies that (recall $d \ge 3$)
\[ \sum_{n \ge n'} \sum_{x \in \M^d} q^2(x,n)   \le \sum_{n \ge n'}    c_2 n^{-d/2}  \le c_3 (n')^{-(d-2)/2} \]
for constants $c_1,c_2, c_3 > 0$ and all $n' \in \N$. Recalling Remark \ref{r:spec}, this shows that the GFF is in class $\F_{(d-2)/2}$.

\smallskip
The proof for the GMM is similar, with the only difference that (recall $d \ge 5$)
\begin{equation*}
\sum_{n \ge n'} \sum_{x \in \M^d} \tilde{q}^2(x,n)  = \sum_{n \ge n'} \frac{n}{2}  \sum_{x \in \M^d} q^2(x,n)   \le  \sum_{n \ge n'} c_2 n^{1-d/2} \le   c_4 (n')^{-(d-4)/2} . \qedhere 
\end{equation*}
\end{proof}


\smallskip

\section{Sharpness of the phase transition}
\label{s:sharp}

In this section we establish the sharpness of the phase transition in~Theorem \ref{t:main1}, and also prove the non-triviality in Proposition~\ref{p:nontriv}.

\smallskip
Throughout this section we fix a Gaussian field $f \in \F_\alpha$, $\alpha > 0$, and a representation $f = q \star_1 W$ (or $f = (q \star_1 W)|_{\Z^d}$ in the discrete case, although we treat continuous and discrete fields simultaneously). Note that the parameter $\alpha > 0$ will appear in several of the bounds, and constants may depend on~it.

\subsection{Preliminaries: The OSSS and Russo-type inequalities}
\label{s:prelim}

As outlined in Section~\ref{s:outline}, the general strategy is to establish the differential inequality \eqref{e:mens3} by combining OSSS and Russo-type inequalities. Here we introduce these inequalities.

\smallskip
Let us immediately specialise to the relevant setting. We work with a dyadic partition of $\R^d \times \R^+$ into disjoint boxes $(B_{i,m})_{ (i,m) \in \Z^d \times \N_0}$, where 
\[ B_{i, m} =  \begin{cases}
\big( i + (0,1]^d \big) \times [0, 1]  &  m = 0 , \\
\big( 2^m i + (0,2^m]^d \big) \times (2^{m-1}, 2^m]  = 2^m  \big( ( i + (0,1]^d ) \times (1/2, 1] \big)    &  m \ge 1  .
  \end{cases} \]
This induces an orthogonal decomposition $W= \sum_{ (i,m) \in \Z^d \times \N_0} W_{i,m}$, where $W_{i,m} = W|_{B_{i,m}}$. We also define $B^+_m = \cup_{(i,m') \in \Z^d \times \N_0 : m' \ge m } B_{i,m'} $ and $W^+_m = W|_{B_m^+}$, and define the random vector
\begin{equation}
\label{e:subset}
 \mathbb{W}_m = \big((W_{i,m'})_{(i,m')  \in \Z^d \times \N_0: m' \le m-1} , W^+_m \big)   . 
 \end{equation}  
 
\smallskip
For an event $A$ and a box $B_{i, m}$, we define the \textit{resampling influence of $W_{i, m}$ on $A$} as
\[ \Inf_{i,m} = \P[  \id_{f \in A} \neq \id_{f' \notin A} ]  ,\] 
where $f'$ is the Gaussian field generated from $f$ by resampling $W_{i,m}$ independently. Equivalently, 
\[ \Inf_{i,m} = 2 \E[ \var[  \id_{f \in A}  | \mathcal{F}_{i,m}  ]   ] ,  \]
 where $\mathcal{F}_{i,m}$ denotes the $\sigma$-algebra generated by $ \sum_{ (i',m') \in \Z^d \times \N_0  \setminus \{(i,m)\} } W_{i',m'}$. Similarly we define
\[ \Inf^+_m = \P[  \id_{f \in A} \neq \id_{f^\ast \notin A} ] = 2 \E[ \var[  \id_{f \in A}  | \mathcal{F}_m^+  ]   ]  , \] 
where $f^\ast$ is the Gaussian field generated from $f$ by resampling $W^+_m$ independently, and $\mathcal{F}_m^+$ denotes the $\sigma$-algebra generated by $ W|_{(B^+_m)^c}$. Note that $\Inf^+_0 = 2 \var[ \id_A ]$, and that $m \mapsto \Inf^+_m$ is non-increasing.

\smallskip
Next, consider a compactly-supported event $A$ (i.e.\ that depends on the restriction of $f$ to a compact domain $D \subset \R^d$) and fix an $m \in \N_0$. A \textit{class-$m$ randomised algorithm that determines $A$} is a random adapted procedure that iteratively reveals the coordinates of the random vector $\mathbb{W}_m$, and terminates once the event $A$ has been determined. Notice that, by the `conic support' property of the kernel $q$, the event $A$ depends on a finite subset of coordinates of $\mathbb{W}_m$, and without loss of generality we may restrict the algorithm to reveal only these coordinates. Fixing such an algorithm, we let $\Rev_{i,m}$ and $\Rev^+_m$ denote the probability that $W_{i,m}$ and $W^+_m$ respectively are revealed by this algorithm. 

\smallskip
The OSSS inequality applied in this setting yields the following:

\begin{proposition}[OSSS inequality {\cite{osss05}}] 
\label{p:osss}
Let $A$ be a compactly-supported event, let $m \in \N_0$, and consider a class-$m$ randomised algorithm that determines $A$. Then
\[ \var[ \id_A ] \le \frac{1}{2}  \Big(  \Rev^+_m \Inf^+_m \ +  \! \! \! \sum_{(i,m') \in \Z^d \times \N_0 : m' \le m-1} \! \! \!  \Rev_{i,m} \Inf_{i,m} \Big) . \]
\end{proposition}
\begin{proof}
This is a direct application of the OSSS inequality, which applies to any finite vector of independent random variables taking values in (possibly distinct) measurable spaces.
\end{proof}

Consider now an \textit{increasing event} $A$, i.e.\ an event $A$ such that $\{f \in A\}$ implies $\{f + h \in A\}$ for every function $h \ge 0$. We next present a Russo-type inequality for $A$, which can be viewed as a consequence of Gaussian isoperimetry. A similar inequality was first established in \cite{dm21} in a related setting, but for completeness we give a proof in Section \ref{s:russo}. We say that an increasing event $A$ is a \textit{continuity event} if  $\ell \mapsto \P_\ell[f + \psi \in A]$ (resp.\ $\ell \mapsto \P_\ell[f|_{\Z^d} + \psi|_{\Z^d}]$ in the discrete case) is continuous for every smooth function $\psi$. Recall the \textit{Dini derivative}, defined for $f: \R \to \R$ as
\[ \frac{d^+}{dx}f(x)=\liminf\limits_{\varepsilon \downarrow 0}\frac{f(x+\varepsilon)-f(x)}{\varepsilon} ,\]
which we use to avoid assuming $\ell\mapsto \P_\ell[A]$ is differentiable.

\begin{proposition}[Russo-type inequality]
\label{p:russo}
There exists a $c > 0$ such that, for every increasing compactly-supported continuity event $A$ and $m \in \N_0$,
\[   \frac{d^+}{d\ell} \P[f + \ell  \in A ] \bigg|_{\ell = 0}  \ge  c (2^m)^{\alpha/2}   \max\Big\{   \Inf^+_m ,  \sum_{i \in \Z^d} \Inf_{i,m}   \Big\} .\]
\end{proposition}

\begin{remark}
We only use the full strength of Proposition \ref{p:russo} in the case $\alpha \le 2(d-1)$; if $\alpha > 2(d-1)$ it is sufficient to have the bound
\[  \frac{d^+}{d\ell} \P[f + \ell  \in A ] \bigg|_{\ell = 0}  \ge  c (2^m)^{\alpha/2}  \sum_{i \in \Z^d} \Inf_{i,m} .\]
\end{remark}

\subsection{The Menshikov-type differential inequality}
Recall the definition of $\theta_R(\ell)$ from~\eqref{e:thetar}. In this subsection we establish that $\theta_R(\ell)$ satisfies the inequality \eqref{e:mens3}:

\begin{proposition}[Menshikov-type differential inequality]
\label{p:mens}
For every $\eps \in (0, \alpha/(2(d-1))$ and $\ell_0 \in \R $ there exists a $c > 0$ such that, for all $ \ell < \ell_0 $ and $R \ge 1$,
 \begin{equation}
\label{e:mens4}
 \frac{d^+}{d\ell} \theta_R(\ell) \ge \frac{c \, \theta_R(\ell)}{ \big( \frac{1}{R} \sum_{r = 1}^{\lceil R \rceil}\theta_r(\ell) \big)^{\gamma}  } ,
 \end{equation}
 where $\gamma = \min\{ 1 , \alpha/(2(d-1)) - \eps \id_{\alpha = 2(d-1)} \}$.
\end{proposition}
\begin{remark}
The $\eps > 0$ in the case $\alpha = 2(d-1)$ is not sharp and is mainly for convenience; one could also replace it with a log-factor.
\end{remark}
\begin{proof}
Fix $\ell_0 \in \R$, $\ell < \ell_0$, $R \ge 1$, and define $A = \{ f + \ell  \in  \Lambda_{1/2} \leftrightarrow \partial \Lambda_R \}$. Recall the resampling influences $\Inf_{i,m}$ and $\Inf^+_m$ from the previous section (relative to $A$). Since $\Inf^+_0 = 2 \var[ \id_A ]$ and $ m \mapsto \Inf^+_m $ is non-increasing, we may define $m_0 \in \N_0 \cup \{\infty\}$ to be the smallest integer satisfying
\[     \Inf^+_m  \ge \var[ \id_A ]  \quad \text{if } m \le m_0    \quad \text{and} \quad  \Inf^+_m  \le  \var[ \id_A ]  \quad \text{if } m > m_0    . \]
We claim that the following two inequalities hold:
\begin{equation}
\label{e:de1}
\frac{d^+}{d\ell} \theta_R(\ell)  \ge     c_1 (2^m)^{\alpha/2}  \theta_R(\ell)     \quad \text{if } m \le m_0  
\end{equation}
and
\begin{equation}
\label{e:de2}
  \frac{d^+}{d\ell} \theta_R(\ell)  \ge     \frac{  c_1  \Big(\sum_{m' \ge 0}^{m-1} (2^{m'})^{d-1-\alpha/2} \Big)^{-1}}{ \frac{1}{R} \sum_{r = 1}^{\lceil R \rceil} \theta_r(\ell) }     \theta_R(\ell)   \quad \text{if } m > m_0 ,
  \end{equation}
  where $c_1 > 0$ is a constant independent of $R$ and $\ell$.
  
 \smallskip
Let us show how \eqref{e:de1} and \eqref{e:de2} imply the proposition. If $m_0 = \infty$ then, taking $m \to \infty$ in \eqref{e:de1}, we have $\frac{d^+}{d\ell} \theta_R(\ell)  = \infty$ and there is nothing to prove. So assume $m_0 < \infty$. We further separate the two cases (i) $\alpha > 2(d-1)$ and (ii) $\alpha \le 2(d-1)$. In the first case we apply \eqref{e:de2} with $m = m_0 + 1$, which yields
  \[   \frac{d^+}{d\ell} \theta_R(\ell)   \ge       \frac{  c_1   \Big(\sum_{m' \ge 0}^{m_0} (2^{m'})^{d-1-\alpha/2} \Big)^{-1} }{ \frac{1}{R} \sum_{r = 1}^{\lceil R \rceil} \theta_r(\ell) }     \theta_R(\ell) \ge      \frac{  c_2   \theta_R(\ell) }{ \frac{1}{R} \sum_{r = 1}^{\lceil R \rceil } \theta_r(\ell) }      \]
  for a constant $c_2 > 0$ independent of $R$ and $\ell$, which completes the proof in that case. In the second case we fix a $\eps \in (0, d-1)$, and apply \eqref{e:de1} with $m = m_0 $ and \eqref{e:de2} with $m = m_0 + 1$. This yields
  \begin{align*}
  \frac{d^+}{d\ell} \theta_R(\ell)  &  \ge c_1 \theta_R(\ell) \max\Big\{   (2^{m_0})^{\alpha/2} ,    \frac{  \Big(\sum_{m' \ge 0}^{m_0} (2^{m'})^{d-1-\alpha/2 } \Big)^{-1} }{ \frac{1}{R} \sum_{r = 1}^{\lceil R \rceil} \theta_r(\ell) }  \Big\}  \\
  & \ge  c_2 \theta_R(\ell) \inf_{t > 0}  \Big\{ t^{\alpha/2} +  \frac{ t^{\alpha/2 - (d-1) - \eps  \id_{\alpha = 2(d-1)}} }{ \frac{1}{R} \sum_{r = 1}^{\lceil R \rceil} \theta_r(\ell) }   \Big\} \\
  & \ge c_3  \theta_R(\ell)  \Big( \frac{1}{R} \sum_{r = 1}^{\lceil R \rceil} \theta_r(\ell) \Big)^{- \alpha/(2(d-1+\eps  \id_{\alpha = 2(d-1)})) }  
  \end{align*}
for constants $c_2, c_3 > 0$ independent of $R$ and $\ell$, which completes the proof.

\smallskip
So let us establish \eqref{e:de1} and \eqref{e:de2}. In fact, \eqref{e:de1} is a direct application of Proposition \ref{p:russo} (which applies to the event $\{\Lambda_{1/2} \leftrightarrow \partial \Lambda_R \}$ by item (4) of Corollary \ref{c:cont}):
\[   \frac{d^+}{d\ell} \theta_R(\ell)    \ge   c_4 (2^m)^{\alpha/2} \Inf^+_m  \ge c_4 (2^m)^{\alpha/2} \theta_R(\ell) (1 - \theta_R(\ell) ) \ge c_5  (2^m)^{\alpha/2} \theta_R(\ell)  \]
where the second inequality used that $ \Inf^+_m  \ge \var[ \id_A ] $ for $m \le m_0$, and the third inequality used that $\theta_R(\ell) \le \theta_1(\ell_0) < 1$ by monotonicity.

\smallskip
It remains to establish \eqref{e:de2}. We use the following claim:

\begin{claim}
\label{c:algo}
For every $m \in \N$ there exists a class-$m$ randomised algorithm determining $A$ such that, for all $m' \le m-1$,
\[ \sup_{i \in \Z^d} \Rev_{i, m'} \le  \frac{c_6 (2^{m'})^{d-1}}{R} \sum_{r = 1}^{\lceil R \rceil} \theta_r(\ell)  \]
where $c_6 > 0$ is a constant independent of $R$ and $\ell$.
\end{claim}

Let us complete the proof of \eqref{e:de2} assuming the claim. Fix $m > m_0$. Applying the OSSS inequality Proposition \ref{p:osss}, and using Claim \ref{c:algo} and the trivial bound $\Rev^+_m \le 1$, we have
\begin{align}
\nonumber \var[ \id_A ] & \le \frac{1}{2} \Big(  \Rev^+_m \Inf^+_m +  \! \!  \sum_{(i,m') \in \Z^d \times \N_0 : m' \le m-1}   \Rev_{i,m} \Inf_{i,m}  \Big) \\
\nonumber & \le  \frac{\Inf^+_m}{2} +  c_6  \Big( \frac{1}{R} \sum_{r = 1}^{\lceil R \rceil}  \theta_r(\ell) \Big) \sum_{m' = 0}^{m-1}  (2^{m'})^{d-1} \sum_{i \in \Z^d} \Inf_{i,m'}  \\
\label{e:sharp1} &   \le  \frac{\Inf^+_m}{2} +    c_7 \Big( \frac{1}{R} \sum_{r = 1}^{\lceil R \rceil}  \theta_r(\ell) \Big) \Big( \frac{d^+}{d\ell} \theta_R(\ell) \Big) \sum_{m' = 0}^{m-1}  (2^{m'})^{d-1 - \alpha/2}  
\end{align}
where in the final inequality we used Proposition \ref{p:russo}. On the other hand, since $m > m_0$,
\begin{equation}
\label{e:sharp2}
  \var[ \id_A ]  - \frac{\Inf^+_m}{2}   \ge  \frac{\var[ \id_A ]}{2}  .
  \end{equation}
Combining \eqref{e:sharp1} and \eqref{e:sharp2}, and using again that $\theta_R(\ell) \le \theta_1(\ell_0) < 1$, establishes the bound.

\smallskip
It remains to prove Claim \ref{c:algo}, which is very similar to claims appearing in \cite{dcrt19a, dcrt20} (and elsewhere). Choose a random integer $r \in \{1, 2, \ldots, \lceil R \rceil\}$, and notice that, for every box $C_i = i + [0,1]^d$, $i \in \Z^d$, the field $f|_{C_i}$ restricted to $C_i$ is determined by (i) $W^+_m$, and (ii) every $W_{i, m'}$, $i \in \Z^d$, $m' \le m-1$, such that $\text{proj}(B_{i,m'})$ is within distance $2^{m'}$ of $C_i$, where $\text{proj}(B_{i,m'})$ is the projection of $B_{i,m'}$ onto $\R^d$. Then by iteratively revealing $f|_{C_i}$ for a sequence of $C_i$, one can explore the connected components of $\{f \le 0\} \cap \Lambda_R$ that intersect $\partial \Lambda_r$, stopping once $A$ is determined. This reveals (i) $W^+_m$, and (ii) a subset of the $W_{i, m'}$, $i \in \Z^d$, $m' \le m-1$, such that $\{ B(\text{proj}(B_{i,m'}) ,2^{m'} + 1) \leftrightarrow  \partial \Lambda_r\}$, where $B(A, s)$ is the set of points within distance $s$ of $A$, and we use the convention that $\{A \leftrightarrow B\}$ occurs if $A$ and $B$ intersect. Averaging over $r$ and by stationarity, the revealment probabilities are at most
\[  \Rev_{i, m'}  \le \frac{2}{R} \sum_{r = 1}^{\lceil R \rceil}   \P_\ell[  \Lambda_{2^{m'+2}} \leftrightarrow  \partial \Lambda_r ]  . \]
Using that $\P[\cdot] \le 1$ and by the union bound, 
\begin{equation}
\label{e:algo1}
\sum_{r = 1}^{\lceil R \rceil}   \P_\ell[  \Lambda_{2^{m'+2}} \leftrightarrow  \partial \Lambda_r ] \le c_8     (2^{m'})^{d-1}  \sum_{r = 1}^{\lceil R \rceil}  \P_\ell[  \Lambda_{1/2} \leftrightarrow  \partial \Lambda_r   ]   ,
\end{equation}
 which gives the claim.
\end{proof}

\subsection{Analysis of the differential inequality}
We next show how \eqref{e:mens4} implies a phase transition with at least stretched-exponential decay of connectivity in the subcritical regime. We rely on the following lemma:

\begin{lemma}
\label{l:mens}
Let $[c_1,c_2] \subset \R$ be an interval, and let $(f_n)_{n \ge 1}$ be a sequence of non-decreasing continuous functions $f_n : [c_1,c_2] \to (0,1]$, which satisfy 
 \begin{equation}
\label{e:mens5}
f_{n+1}(x) \le f_n(x)  \quad \text{and} \quad \frac{d^+}{dx} f_n(x) \ge \frac{\kappa f_n(x)}{ \big( \frac{1}{n} \sum_{i = 1}^n f_i(x) \big)^{\gamma}  } \ , \quad n \ge 1, x \in [c_1,c_2],
 \end{equation}
 where $\kappa > 0$ and $\gamma \in (0,1]$ are constants. Define $f = \lim_{n \to \infty} f_n$ and $x_c = \inf\{ x : f(x) > 0 \}$. Then:
 \begin{itemize}
\item For every $x < x_c$ there exists a $c_3 > 0$ such that, for every $n \ge 1$,
\begin{equation}
\label{e:mens6}
 f_n(x) \le  e^{-c_3 n^\gamma  }.
 \end{equation}
\item For every $x > x_c$, 
\[ f(x) \ge  (\kappa  \gamma)^{1/\gamma} ( x - x_c)^{1/\gamma}  . \]
\end{itemize}
\end{lemma}

\begin{proof}[Proof of Lemma \ref{l:mens}]
By multiplying $f_n$ by a constant, it suffices to consider the case $\kappa = 1$. We prove the two items separately.

\smallskip
\noindent \textit{The case $x < x_c$.} We adapt the analysis of \eqref{e:mens1} in \cite{men86} (see also \cite[Section 5.2]{gr99}); the idea is to start from the knowledge that $f_n(x) \to 0$ as $n \to \infty$, and use \eqref{e:mens5} to iteratively bootstrap this estimate for smaller~$x$. So fix some $x < x'  < x_c$. We proceed in two steps:

\smallskip
\noindent \textit{Step 1: Polynomial decay at $x'$.} By integrating \eqref{e:mens5} and using monotonicity we have
\begin{equation}
\label{e:mens7}
 f_n(x) \le f_n(y) e^{   - (y-x) / ( \frac{1}{n} \sum_{i = 1}^n f_i(y)  )^\gamma } 
 \end{equation}
 for every $x < y$ and $n \ge 1$. Now choose a $x_0 \in (x' ,x_c)$, a large integer $n_0 \ge 1$ be to determined later, and define sequences $(x_i)_{i \ge 1}$ and $(n_i)_{i \ge 1}$, decreasing and increasing respectively, via
 \[ x_{i+1} = x_i - 3^\gamma g_i^\gamma \log (1/g_i) \quad \text{and} \quad n_{i+1} = n_i \lfloor (1/g_i) \rfloor  \]
 where $g_i = f_{n_i}(x_i) \in (0,1]$ is non-increasing. In particular, since $f_n(x_0) \to 0$ as $n \to \infty$, we can make $g_0$ arbitrarily small by choosing $n_0$ sufficiently large. Now, since $f_n(x) \le 1$ and by monotonicity,
\[\frac{1}{n_{i+1}} \  \sum_{j = 1}^{n_{i+1}} f_j(x_i)  \le \frac{n_i}{n_{i+1}} + f_{n_i}(x_i) \le 1/(1/g_i-1) + g_i \le 2 g_i /(1-g_i)  , \]
and so applying \eqref{e:mens7} over $(x_{i+1},x_i)$ and using monotonicity again, we have
 \begin{align*}
  g_{i+1} = f_{n_{i+1}}(x_{i+1}) & \le f_{n_{i+1}} (x_i) e^{   - (x_i - x_{i+1}) / ( \frac{1}{n_{i+1}} \sum_{j = 1}^{n_{i+1}} f_j(x_i)  )^\gamma }  \\
  & \le  g_i e^{- 3^\gamma  2^{-\gamma} (1-g_i)^\gamma \log(1/g_i) }  \le g_i^2 
  \end{align*}
 where the final inequality holds provided $g_0 \le 1/3$. A consequence of $g_{i+1} \le g_i^2$ is that
 \[ \lim_{i \to \infty} x_i \ge x_0 - 3^\gamma  \sum_{i \ge 0 } g_i^\gamma \log (1/g_i)  \ge  x_0 - 3^\gamma   \sum_{i \ge 0 } (g_0^{\gamma/2})^{2^i}> x' \]
provided $g_0$ is sufficiently small. A second consequence is that 
\[ g_{i-1}^2  \le g_{i-1} g_{i-2} \ldots g_1 g_0^2 \le g_0 n_0 / n_i \]
where we used $g_{j+1} \le g_j^2$ in the first inequality, and $n_{j+1} \le n_j / g_j$ in the second. Combining these two observations and using monotonicity, for every $n_{i-1} \le n \le n_i$ we have
\[ f_n(x') \le f_{n_{i-1}}(x_{i-1}) = g_{i-1} \le (g_0 n_0)^{1/2}  n_i^{-1/2} \le (g_0 n_0)^{1/2}  n^{-1/2} .\]

\smallskip
\noindent \textit{Step 2: Stretched-exponential decay at $x$.} 
By the previous step,
\[  \frac{1}{n} \sum_{i = 1}^n f_i(x')  \le c_3 n^{-1/2}    \]
for some $c_3 > 0$ and all $n \ge 1$. Fixing a $x'' \in (x,x')$ and applying \eqref{e:mens7} over $(x'',x')$, we deduce that 
\[  f_n(x'') \le  e^{-c_4 n^{\gamma/2} } , \]
 and hence $   \sum_{i \ge 1} f_i(x'')  \le c_5$, for some $c_4,c_5 > 0$. Applying \eqref{e:mens7} once more over $(x,x'')$, we deduce \eqref{e:mens6}. 

\smallskip
\noindent \textit{The case $x > x_c$.} The analysis of \eqref{e:mens1} in \cite[Section 5.2]{gr99} does not go through if $\gamma < 1$ so we give a new argument valid for all $\gamma \in (0,1]$. Fix $x_c < x' < x$ and $\eps, \delta> 0$. By Egorov's theorem there exists a set $E_\eps$ of Lebesgue measure $\eps$ such that $f_n \to f$ uniformly on $[x',x] \setminus E_\eps$. Since $f$ is non-decreasing, and $f_n(y) \ge f(x') > 0$ for all $y \in [x',x]$, this implies we may choose $n_0 = n_0(x,x',\eps,\delta) > 0$ sufficiently large so that
\[ \frac{1}{n} \sum_{i = 1}^n f_i(y)  \le (1+\delta) f_n(y)  \ ,  \quad n \ge n_0 , \ y \in [x',x] \setminus E_\eps. \]
Combining with \eqref{e:mens5}, we have the inequality
\begin{equation}
\label{e:mens8}
 \frac{d^+}{dy} f_n(y) \ge  (1+\delta)^{-\gamma} f_n(y)^{1-\gamma}  \ , \quad  n \ge n_0 , \ y \in [x',x] \setminus E_\eps. 
 \end{equation}
Integrating \eqref{e:mens8} over $[x',x] \setminus E_\eps$, and since both $f_n$ and $y \mapsto f_n(y)^{1-\gamma}$ are non-decreasing (recall that $\gamma \in (0,1]$), we deduce that
\[ f_n(x) - f_n(x') \ge (1+\delta)^{-1} \gamma^{1/\gamma}(x-x' - \eps)^{1/\gamma} \ , \quad n \ge n_0 . \]
Taking $n \to \infty$ and then $\eps,\delta \to 0$ shows that
\[f(x) - f(x') \ge   \gamma^{1/\gamma}(x-x')^{1/\gamma}  ,   \]
and taking $x' \downarrow x_c$, and since $f \ge 0$, we conclude the result.
\end{proof}

\begin{remark}
\label{r:alpha0}
The proof of Lemma \ref{l:mens} can be adapted to analyse the general inequality
\[ \frac{d^+}{dx} f_n(x) \ge \frac{\kappa f_n(x)}{ \varphi( \frac{1}{n} \sum_{i = 1}^n f_i(x) \big)  }   , \]
 at least if $\varphi(x) \ge (\log (1/x))^{-(1+\delta)}$ for some $\delta > 0$ and sufficiently small $x > 0$. In that case the decay is faster-than-polynomial in the regime $x < x_c$, but not always stretched-exponential. This could be useful to study models with slower-than-polynomial correlation decay (i.e.\ the $\alpha = 0$ case of Theorem \ref{t:main1}).
\end{remark}

\begin{remark}
The conclusion of Lemma \ref{l:mens} actually holds for all $\gamma > 0$. Although our proof of the second item used that $\gamma \in (0,1]$, the case $\gamma \ge 1$ can be treated using a similar argument to in \cite[Section 5.2]{gr99}.

\smallskip
In \cite{dcrt19a} a new argument was introduced to deduce sharpness from~\eqref{e:mens3} in the case $\gamma = 1$. We were unable to adapt this to $\gamma < 1$, although it would be interesting to do so.
\end{remark}

\subsection{Bootstrapping and non-triviality}

At this point we could already combine Proposition \ref{p:mens} and Lemma \ref{l:mens} to conclude the proof of the second item of Theorem \ref{t:main1} and sharpness in the sense of Corollary~\ref{c:main1}, as well as the first item in the case $\alpha > 2(d-1)$. However to obtain the full strength of the first item of Theorem \ref{t:main1} in the case $\alpha \le 2(d-1)$, as well as Proposition \ref{p:nontriv}, we need an additional bootstrapping step.

\smallskip
We will use the following general result from \cite{ms22}, whose proof uses sprinkled bootstrapping arguments pioneered in \cite{rs13,pr15,pt15}. Recall the critical parameters $\ell_c$ and $\ell^\ast_c$ (from Corollary~\ref{c:main1}); here we extend their definition to general random fields $f$ on $\R^d$ or $\Z^d$.

\begin{proposition}[See {\cite[Proposition 4.1]{ms22}}]
\label{p:boot}
Let $f$ be a continuous stationary random field on $\R^d$ such that, for every $R \ge 1$, 
\[ f \stackrel{d}{=} f_R + g_R, \]
where $f_R$ is a stationary $R$-range dependent field on $\R^d$, and $g_R$ is a continuous field on $\R^d$ satisfying, for constants $\delta, \psi, c_1> 0 $, 
\begin{equation}
\label{e:tail}
 \P\big[ \sup_{x \in \Lambda_R} |g_R(x)| \ge  (\log R)^{-(1+\delta)} \big] \le  e^{- c_1R^\psi } \ , \quad R \ge 2 . 
 \end{equation}
 Then
\[  \ell^\ast_c(f) > - \infty . \]
Moreover, for every $\ell < \ell^\ast_c$ and $\psi' \in (0, \psi) \cap (0, 1]$, there exists a $c_2> 0$ such that
\begin{equation}
\label{e:armbound}
 \P[ \Lambda_{1/2} \leftrightarrow \partial \Lambda_R ] \le e^{- c_2 R^{\psi'}} \ , \quad R \ge 1.
 \end{equation}
The conclusion also holds if `continuous field on $\R^d$' is replaced by `field on~$\Z^d$'.
\end{proposition}

\begin{proof}
For continuous fields this is a special case of \cite[Proposition 4.1]{ms22}, and the proof for discrete fields is identical.
\end{proof}

Using duality we can obtain from \eqref{e:armbound} that $\ell_c < \infty$:

\begin{corollary}
Let $f$ be as in Proposition \ref{p:boot} and suppose that, in addition, $f$ and $-f$ are equal in law. Then $\ell_c(f) < \infty$.
\end{corollary}
\begin{proof}
Let us first consider continuous fields. We aim to show that $\ell_c \le -\ell^\ast_c < \infty$. Choose some $\ell > -\ell^\ast_c$, and let $E_n$ and $F_n$ be respectively the events that $\{f|_{\R^2} \le 0\} \cap ([0,2^{n+1}] \times [0, 2^n])$ contains a path that crosses $[0,2^{n+1}] \times [0, 2^n])$ from left-to-right, and $\{f|_{\R^2} \le 0\} \cap ([0,2^n] \times [0, 2^{n+1}])$ contains a path that crosses $[0,2^n] \times [0, 2^{n+1}]$ from top-to-bottom. For any $n_0 > 0$,
\[ \cap_{n \ge n_0} E_n \cap F_n \] 
implies an infinite component in $\{f \le 0\} \cap \R^2$, and so by the Borel-Cantelli lemma it remains to prove that $\max\{ \P_\ell[  E_n^c],\P_\ell[ F_n^c]  \}$ is summable over $n \ge 1$, since that will imply that $\ell > \ell_c$ and hence $\ell_c \le -\ell^\ast_c$.

\smallskip
Since $f$ and $-f$ are equal in law and by the union bound,
\begin{equation}
\label{e:nontriv1}
  \P_\ell[ E_n^c] \le  c_1 2^{n+1} \P_{-\ell}[  \Lambda_{1/2} \longleftrightarrow \partial \Lambda_{2^n} ] 
  \end{equation}
and similarly for $F_n$. Since $-\ell < \ell^\ast_c$, and by the bound \eqref{e:boot}, 
\[\max\{ \P_\ell[  E_n^c],\P_\ell[ F_n^c]  \} \le c_2 2^{n+1} e^{- c_3 (\log 2^n)^{1+\delta}}  , \]
which gives the required summability.

\smallskip
In the discrete case one needs to proceed slightly differently due to a lack of self-duality. Specifically, one defines the critical parameter 
\[ \hat{\ell}^\ast_c  = \inf \big\{ \ell : \liminf_{R \to \infty} \P_\ell[ \Lambda_R \stackrel{\ast}{\leftrightarrow} \partial \Lambda_{2R} ]  > 0 \big\}   , \]
where $\{A  \stackrel{\ast}{\leftrightarrow} B\}$ is the event that $A$ and $B$ are `$\ast$-connected', i.e.\ connected by a path of vertices in $\Z^d$ within sup-distance $1$. Then the proof of \cite[Proposition 4.1]{ms22} adapts immediately to show that, for sufficiently small $\ell$, any $\psi \in (0, \psi) \cap (0, 1]$, and a constant $c_3 > 0$,
 \begin{equation}
\label{e:boot}
 \P_{\ell}[  \Lambda_R \stackrel{\ast}{\leftrightarrow} \partial \Lambda_{2R}   ]   \le  e^{- c_3 R^{\psi'}} \ , \quad R \ge 1.
 \end{equation}
This shows in particular that $\hat{\ell}^\ast_c > -\infty$, and it remains to show that $\ell_c \le -\hat{\ell}^\ast_c$. So choose some $\ell > -\hat{\ell}^\ast_c$, and let $E_n$ and $F_n$ be respectively the events that $\{f|_{\Z^2} \le 0\} \cap ([0,2^{n+1}] \times [0, 2^n])$ contains a path that crosses $[0,2^{n+1}] \times [0, 2^n])$ from left-to-right, and $\{f|_{\Z^2} \le 0\} \cap ([0,2^n] \times [0, 2^{n+1}])$ contains a path that crosses $[0,2^n] \times [0, 2^{n+1}]$ from top-to-bottom. For any $n_0 > 0$,
\[ \cap_{n \ge n_0} E_n \cap F_n \] 
implies an infinite component in $\{f \le 0\}$. On the other hand, since $f$ and $-f$ are equal in law and by the union bound,
\[  \P_\ell[ E_n^c] \le  2^{n+1} \P_{-\ell}[  \Lambda_{1/2} \stackrel{\ast}{\longleftrightarrow} \partial \Lambda_{2^n} ] \]
and similarly for $F_n$. By the Borel-Cantelli lemma, $\max\{ \P_\ell[  E_n^c],\P_\ell[ F_n^c]  \}$ is summable, which implies $\ell > \ell_c$ and hence $\ell_c \le -\hat{\ell}^\ast_c$.
 \end{proof}

\begin{remark}
\label{r:supercrit}
Using similar arguments one can actually show that there exists a $\tilde{\ell}_c < \infty$ such that if $\ell > \tilde{\ell}_c$ the model is in a `strongly supercritical' regime, see e.g.\ \cite{rs13, drs14, dgrs20, cn21}. However, unlike the proof of sharpness for the GFF in \cite{dgrs20}, our approach does \textit{not} allow us to conclude the supercritical analogue of sharpness $\tilde{\ell} = \ell_c$, which remains an important open problem for all models except the GFF.
\end{remark}

\subsection{Conclusion of the proof}
We are now ready to conclude the proof of sharpness:

\begin{proof}[Proof of Proposition \ref{p:nontriv} and Theorem \ref{t:main1}]
We first claim that $f \in \F$ satisfies the assumptions of Proposition~\ref{p:boot} for any $\psi \in (0,\alpha)$ and $c,\delta > 0$. Recall the representation $f = q \star_1 W$. For $R > 0$, $f$ can be decomposed as $f = f_R + g_R$, where
\[ f_R = \int_{\R^d \times (0, R]} q(x-y, t) dW(y,t)   \quad \text{and} \quad  g_R =   \int_{\R^d \times (R,  \infty)} q(x-y, t) dW(y,t) .\]
Since $q(x, t)$ is supported on $\{ |x| \le t/2\}$, $f_R$ is $R$-range dependent, and it remains to show that $g_R$ satisfies the tail bound \eqref{e:tail}.

\smallskip
We treat the case of continuous and discrete fields slightly differently. In the continuous case, consider the field $h(x) = R^{-\alpha/2} g_R(x)$. By the assumption on $q$ in \eqref{e:fdecay}, and since $\partial^k q \in L^2(\R^d \times \R^+)$ by Assumption \ref{a:main}, for all $R \ge 1$ and every multi-index $k$ with $|k| \le 1$,
\[ \var[\partial^k h(0)] =  R^{\alpha \id_{k=0}} \int_{\R^d \times (R, \infty)} (\partial^k q)^2(x-y, t)  \, dx dt \le c_1  \]
for a constant $c_2  > 0$. Therefore by Kolmogorov's theorem (see \cite[Appendix A.9]{ns16}) 
\[  \E \Big[ \sup_{x \in \Lambda_1} h(x) \Big]   \le  c_2  .   \]
By rescaling and applying the BTIS theorem (see \cite[Theorem 2.9]{aw09}), for any $c , \delta > 0$,
\begin{align*}
 \P \Big[ \sup_{x \in \Lambda_1} g_R(x) \ge    (\log R)^{-(1+\delta)}    \Big]  & =  \P \Big[ \sup_{x \in \Lambda_1} h(x) \ge    (\log R)^{-(1+\delta)} R^{\alpha/2}   \Big] \\ 
 & \le    e^{-  c_3 R^\alpha (\log R)^{-2(1+\delta)}}    \le e^{- c_4 R^{-\psi}} 
 \end{align*}
for constants $c_3, c_4> 0$ and sufficiently large $R$. Eq.\ \eqref{e:tail} then follows from union bound.

\smallskip
The case of discrete fields is similar but simpler: the assumption $ \eqref{e:fdecay}$ directly implies that   $\var[g_R(0)] \le c_5 R^{-\alpha}$, and so
\[  \P \Big[  g_R(0) \ge    (\log R)^{-(1+\delta)}    \Big]   \le    e^{-  c_6 R^\alpha (\log R)^{-2(1+\delta)}}    \le e^{- c_7 R^{-\psi}}  \] for sufficiently large $R$, and then \eqref{e:tail} again follows from union bound.

\smallskip
Proposition \ref{p:nontriv} is then an immediate consequence of Proposition \ref{p:boot}. Moreover, for any $\ell_0 > 0$ such that $-\ell_0 < \ell_c < \ell_0$, by Proposition \ref{p:mens} and item (4) of Corollary \ref{c:cont} the assumptions of Lemma \ref{l:mens} are satisfied for the function $\theta_R(\ell)$ restricted to $R \in \N$ and $\ell \in [\ell_0',\ell_0]$. From the conclusion of Lemma \ref{l:mens} we deduce that $\ell^\ast_c = \ell_c$ (see the proof of Corollary \ref{c:main1}), and Theorem~\ref{t:main1} then follows by a second application of Proposition~\ref{p:boot} and monotonicity.
\end{proof}

\subsection{Proof of the Russo-type inequality}
\label{s:russo}
To conclude the section we prove Proposition~\ref{p:russo}. The main ingredient is the following statement of Gaussian isoperimetry. For a set $E \subset\mathbb{R}^n$, let
\[ E^{+\varepsilon} := \{ x \in \mathbb{R}^n : \text{ there exists } w \in E \text{ s.t. } \|x-w\|_{L^2} \le \varepsilon \} \]
denote the $\varepsilon$-thickening (in $L^2$) of $E$.

\begin{proposition}[Gaussian isoperimetry; see {\cite[Proposition 4.2]{dm21}}]
\label{p:iso}
There exists a universal constant $c > 0$ such that, for every $n \in \N$, every measurable $E \subset\mathbb{R}^n$, and all $\varepsilon \ge 0$,
\[ \mathbb{P}[X \in E^{+\varepsilon} \setminus E] \ge   \sqrt{\frac{2}{\pi}}    \mathbb{P}[X \in E] (1 - \mathbb{P}[X \in E]) \varepsilon    -  c \varepsilon^2  ,  \]
where $X$ is an $n$-dimensional standard Gaussian vector.
\end{proposition}
 
\begin{proof}[Proof of Proposition \ref{p:russo}]
In the proof $c_i > 0$ will denote constants that depend neither on the event $A$ nor on $m \in \N_0$.

\smallskip
Let us first show that
\begin{equation}
\label{e:russo1}
    \frac{d^+}{d\ell} \P[f + \ell  \in A ] \bigg|_{\ell = 0}  \ge  c_1 (2^m)^{\alpha/2} \Inf^+_{m}  ,
    \end{equation}
Recall the definitions of $B^+_m$, $W_m^+$, and $\mathcal{F}_m^+$ from Section \ref{s:prelim}. Define $f_m = q \star_1 W_m^+$, let $f'_m$ denote an independent copy of $f_m$, and define $h_m = f - f_m$. We claim that, $\mathcal{F}_m^+$-almost surely,
\begin{equation}
\label{e:rusfin}
\frac{d^+}{d \ell} \mathbb{P}\big[f_m + h_m +  \ell \in A \big| \mathcal{F}_m^+ \big] \Big|_{\ell = 0}   \ge   c_2   (2^m)^{\alpha/2}   \mathbb{P}\big[ \id_{\{f_m + h_m  \in A\}} \neq \id_{\{f'_m + h_m \in A \}} \big| \mathcal{F}_m^+ \big]  .
\end{equation}
This concludes the proof of \eqref{e:russo1} since
\begin{align*}
 \frac{d^+}{d \ell} \mathbb{P}[f_m + \ell  \in A]  \Big|_{\ell = 0} & \ge  \E \Big[ \frac{d^+}{d \ell} \mathbb{P}[f_m + h_m +  \ell  \in A | \mathcal{F}_m^+ ] \Big|_{\ell = 0}  \Big]  \\
 & \ge c_2  (2^m)^{\alpha/2} \E \big[      \mathbb{P}[ \id_{\{f_m + h_m  \in A\}} \neq \id_{\{f'_m + h_m \in A \}} | \mathcal{F}_m^+  ]  \big] \\
 & =:  (c_2/2)  (2^m)^{\alpha/2} \Inf_m^+ 
 \end{align*}
where the first inequality is Fatou's lemma, and the second inequality is by \eqref{e:rusfin}.

\smallskip
So let us prove \eqref{e:rusfin}. Henceforth we condition on $\mathcal{F}_m^+$ and drop it from the notation. Let $(\varphi_j)_{j \in \N}$ be an orthonormal basis of $L^2(B_m^+)$, and let $Z = (Z_j)_{j \in \N}$ be a sequence of i.i.d.\ standard Gaussians. Then
\begin{equation}
\label{e:orthodecomp}
 f^n_m := \sum_{j \ge 1}^n Z_j (q \star_1 \varphi_j) \Rightarrow f_m ,  
 \end{equation}
in the sense of finite-dimensional distributions, and also in law in the $C^0(\R^d)$-on-compacts topology if $f$ is continuous (see Lemma \ref{l:c0conv}). Fixing $n \in \mathbb{N}$ and viewing $\{f_m^n + h_m \in A\}$ as a Borel set $E$ in the $n$-dimensional Gaussian space generated by $Z^n = (Z_j)_{j \le n}$, by Proposition~\ref{p:iso}
\begin{equation}
\label{e:rus2}
   \mathbb{P}[Z^n \in E^{+\varepsilon} \setminus E]   \ge   c_3  \varepsilon   \mathbb{P}[Z^n \in E] (1 - \mathbb{P}[Z^n \in E])  -  c_4 \varepsilon^2  
   \end{equation}
for some $c_3, c_4 > 0$ and every $\varepsilon \ge 0$. Consider $w  = (w_j) \in \mathbb{R}^n$ such that $\| w \|_{L^2} = \varepsilon > 0$. By the Cauchy-Schwarz inequality, for every $x \in \R^d$,
\begin{align*}
   \sum_{j \le n} w_j  (q \star_1   \varphi_j)(x)   =    \Big( q \star_1 \Big( \sum_{j \le n} w_j    \varphi_i \Big) \Big)(x)  & \le \|q\|_{L^2(\R^d \times [2^{m-1} \id_{m > 0},\infty))} \times \Big\| \sum_{j \le n} w_j \varphi_j \Big\|_{L^2}\\
  & \le   c_5 (2^m)^{-\alpha/2}  \|w\|_{L^2} = c_5 (2^m)^{-\alpha/2} \varepsilon ,
  \end{align*}
  where the last inequality used \eqref{e:fdecay} and the fact that $\varphi_j$ is an orthonormal basis. We therefore have, for every $x \in \R^d$,
\begin{align*}
&   \sup_{w : \|w\|_{L^2} \le \varepsilon}  \, \sum_{i \le n} (Z_j +w_j )( q \star_1 \varphi_j )(x)  - f^n_m(x) \\
  & \qquad =   \sup_{w : \|w\|_{L^2} \le \varepsilon}  \,   \sum_{j \le n} w_j  (q \star_1   \varphi_j)(x)  \le  c_5 (2^m)^{-\alpha/2} \eps . 
  \end{align*}
Therefore, since $A$ is increasing and using \eqref{e:rus2}, for every $\eps > 0$,
\begin{align}
  \label{e:rusfin2} 
 &   \mathbb{P}[f^n_m +  h_m + c_5 (2^m)^{-\alpha/2}  \varepsilon  \in A]  - \mathbb{P}[ f^n_m + h_m \in A ]  \\
 \nonumber & \qquad \qquad  \ge    \mathbb{P}[\cup_{w : \|w\|_{L^2} \le \varepsilon}  \{ Z^n + w \in E \} ] - \mathbb{P}[Z^n \in E]  \\
  \nonumber  & \qquad \qquad   =      \mathbb{P}[Z^n \in E^{+\varepsilon} \setminus E ]   \\
 \nonumber &    \qquad \qquad  \ge   c_3   \varepsilon   \mathbb{P}[f^n_i + h_i \in A] (1 - \mathbb{P}[f^n_i + h_i \in A])  - c_4 \varepsilon^2   .
    \end{align}
   To finish, recall that $A$ is an increasing compactly-supported continuity event, which means that, for almost every $f = f_m + h_m$, there exists $\delta > 0$ such that  
\[   \id_{\{f_m + h_m  + s \in A\}}  \qquad \text{and} \qquad   \id_{\{ f_m + h_m + \varepsilon + s \in A \}}  \]
are constant for $s \in (-\delta,\delta)$. In particular, since $f^n_m \to f_m$ in the sense of finite-dimensional distributions, and also in law in the $C^0(\R^d)$-on-compacts topology if $f$ is continuous, $\mathcal{F}_m^+$-almost surely
\begin{equation}
\label{e:cont1}   
\mathbb{P}[f^n_m + h_m \in A] \to   \mathbb{P}[f_m + h_m \in A]  \quad \text{and} \quad  \mathbb{P}[f^n_m +  h_m + \varepsilon  \in A]  \to  \mathbb{P}[f_m +  h_m + \varepsilon  \in A]  . 
\end{equation} 
Then sending $n \to \infty$ in \eqref{e:rusfin2}, and then $\eps \to 0$, gives \eqref{e:rusfin}.

\smallskip
 Let us next show that
\begin{equation}
\label{e:russo2}
    \frac{d^+}{d\ell} \P[f + \ell  \in A ] \bigg|_{\ell = 0}  \ge  c_1 (2^m)^{\alpha/2}  \sum_{i \in \Z^d} \Inf_{i,m} .
    \end{equation}
 The proof is similar to \eqref{e:russo1}. For each $i \in \Z^d$, let $\psi_i : \R^d \to [0,1]$ be a smooth function such that $\psi_i(x) = 1$ on $\{x : d_\infty(x, \text{proj}(B_{i,m})) \le 2^{m+1}\}$ and $\psi_i(x) = 0$ on $\{x : d_\infty(x, \text{proj}(B_{i,m})) \ge 2^{m+2}\}$, recalling that $\text{proj}(B_{i,m})$ denotes the projection of $B_{i,m}$ onto $\R^d$. Then $\sum_{i \in \Z^d}\psi_i(x) \le c_2$, for $c_2 > 0$ depending only on $d$. Therefore, since $A$ is increasing, and by the multivariate chain rule for Dini derivatives,
\begin{equation}
\label{e:rusfin3}
 \frac{d^+}{d \ell} \mathbb{P}[f + \ell \in A]  \Big|_{\ell = 0} \ge c_2^{-1} \sum_{i \in \Z^d} \frac{d^+}{d \ell} \mathbb{P}[f + \ell \psi_i \in A]  \Big|_{\ell = 0} .
 \end{equation}
 For each $i \in \Z^d$, let $f_i = q \star_1 W_{i,m}$, let $f'_i$ denote an independent copy of $f_i$, and define $h_i = f - f_i$. Recall that $\mathcal{F}_{i,m}$ denotes the $\sigma$-algebra generated by $ \sum_{ (i',m') \in \Z^d \times \N  \setminus (i,m) } W_{i',m'}$. We claim that, $\mathcal{F}_{i,m}$-almost surely
\begin{equation}
\label{e:rusfin4}
\frac{d^+}{d \ell} \mathbb{P}[f_i + h_i +  \ell \psi_i \in A | \mathcal{F}_{i,m} ] \Big|_{\ell = 0}   \ge     c_3  (2^m)^{\alpha/2}    \mathbb{P}[ \id_{\{f_i + h_i  \in A\}} \neq \id_{\{f'_S + h_i \in A \}} | \mathcal{F}_{i,m} ]  .
\end{equation}
Assuming this, and combining with \eqref{e:rusfin3}, the proof of \eqref{e:russo2} follows in the same way as for~\eqref{e:russo1}.

\smallskip
The proof of \eqref{e:rusfin4} is analogous to \eqref{e:rusfin}. Fix $i \in \Z^d$, condition on $\mathcal{F}_{i,m}$ and drop it from the notation. Let $(\varphi_j)_{j \ge 1}$ be an orthonormal basis of $L^2(B_{i,m})$, and let $Z = (Z_j)_{j \ge 1}$ be a sequence of i.i.d.\ standard Gaussians. Fix $n \in \mathbb{N}$ and consider $w  = (w_j) \in \mathbb{R}^n$ such that $\| w \|_{L^2} = \varepsilon > 0$. Then by the Cauchy-Schwartz inequality and using \eqref{e:fdecay} as before,
\begin{align*}
  \Big\|   \sum_{j \le n} w_i  (q \star_1   \varphi_j)   \Big\|_\infty  & =  \Big\|  q \star_1 \Big( \sum_{j \le n} w_j   \varphi_j \Big)   \Big\|_\infty  \\
  & \le \|q \|_{L^2( \R^d \times [2^{m-1} \id_{m > 0}, 2^m]  )} \Big\| \sum_{j \le n} w_j \varphi_j \Big\|_{L^2}  \le c_5 (2^m)^{-\alpha/2}  \varepsilon .
  \end{align*}
 Recalling that $q$ has conic support, $q \star \varphi_j$ is supported on  $\{x : d_\infty(x, \text{proj}(B_{i,m})) \le 2^{m+1} \}$. Since also $\psi_i(\cdot) \ge \id_{d_\infty(\cdot, \text{proj}(B_{i,m})) \le 2^{m+1}}$, we therefore have, for every $x \in \R^d$,
\[   \sup_{w : \|w\|_{L^2} \le \varepsilon}  \, \sum_{j \le n} (Z_j +w_j )( q \star \varphi_j )(x) - f^n_i =   \sup_{w : \|w\|_{L^2} \le \varepsilon}  \,   \sum_{j \le n} w_j  (q \star   \varphi_j)(x)   \le  c_5 (2^m)^{-\alpha/2}  \varepsilon  \psi_i . \]
Therefore, since $A$ is increasing,
\begin{align*}
    & \mathbb{P}[f^n_i +  h_i +  c_5 (2^m)^{-\alpha/2}  \varepsilon  \psi_i \in A]  - \mathbb{P}[ f^n_i + h_i \in A ]   \\
    & \qquad  \ge    \mathbb{P}[\cup_{w : \|w\|_{L^2} \le \varepsilon}  \{ Z^n + y \in E \} ] - \mathbb{P}[Z^n \in E]  \\
    &  \qquad  =      \mathbb{P}[Z^n \in E^{+\varepsilon} \setminus E ]  .  
  \end{align*}
  This establishes \eqref{e:rusfin4} by taking $n \to \infty$ and $\eps \to 0$, as before.
  \end{proof}

\smallskip
\section{Weak mixing and applications to nodal connectivity of planar fields}
\label{s:wm}

In this section we consider a continuous Gaussian field $f$ on $\R^d$ in class $\calS$, with~$K$ its covariance kernel. We begin by studying the reproducing kernel Hilbert space (RKHS) of $f$, establishing the existence of certain functions with desirable properties. We then show how the existence of these functions implies a weak mixing property for $f$. Finally we show that, in the planar case, the weak mixing property implies Theorem \ref{t:main3} and Corollary~\ref{c:main3}, and also the stronger density bound in Theorem~\ref{t:main2}.

\subsection{The reproducing kernel Hilbert space}
Fix a smooth function $\psi: \R^+ \to [0, 1]$ with $\psi(x) = 1$ for $x \le 1/4$ and $\psi(x) = 0$ for $x \ge 1/2$. For $R > 0$ define 
\[ q_1(x, t) = \sqrt{w(t)} \psi(|x|/R) Q(|x|/t) \quad \text{and} \quad q_2(x,t) = \sqrt{w(t)} (1-\psi(|x|/R)) Q(|x|/t) , \]
and define $f_R = q_1 \star_1 W$ and $g_R = q_2 \star_1 W$. As shown in the proof of Proposition~\ref{p:smprop}, $f$ can be decomposed as $f = f_R + g_R$, where $f_R$ is a stationary $R$-range dependent Gaussian field, and $g_R$ is a stationary Gaussian field satisfying the bound \eqref{e:gbound}. 

\smallskip
Associated to the pair $(f_R, g_R)$ is the RKHS space $H_R \subseteq C^0(\R^d) \times C^0(\R^d)$ defined as
\[   H_R = \big\{  ( q_1 \star_1 \varphi , q_2 \star_1 \varphi) \big\}_{ \varphi \in L^2(\R^d \times \R^+)}  \]
equipped with the inner product inherited from $L^2(\R^d \times \R^+)$  (see Section \ref{s:rkhs}). Our weak mixing property will follow from the existence of $h  = (h_1,h_2) \in H_R$, with $\|h\|_{H_R}$ not too large, such that $h_1$ is positive and large on a domain $D_1$, negative and large on a disjoint domain $D_2$, and for which $h_2$ has small magnitude. 

\begin{proposition}
\label{p:rkhsbound}
For every pair of disjoint piecewise smooth Lipschitz domains $D_1$ and $D_2$, there exists a $c > 0$ such that, for all $f \in \calS_\alpha$, $\alpha > 0$,
\begin{equation}
\label{e:rkhsbound1}
\limsup_{R \to \infty}  R^{-d}  \, \Delta_R(RD_1, RD_2)  < \infty   ,
 \end{equation} 
 and if $\alpha \le d$, 
 \begin{equation}
\label{e:rkhsbound2}
\limsup_{R \to \infty}  R^{-d-1} w(R)^{-1}   \, \Delta_R(RD_1, RD_2) \le c  ,
 \end{equation} 
 where
\begin{equation}
\label{e:inf}
 \Delta_R(S_1,S_2) =  \inf\big\{  \| (h_1,h_2) \|_{H_R}^2 : h_1|_{S_1} \ge 1, h_1|_{S_2} \le -1,  \| h_2 \|_\infty  \le 1/2 \big\} .  
 \end{equation}
 \end{proposition}

\begin{remark}
The two cases \eqref{e:rkhsbound1} and \eqref{e:rkhsbound2} correspond to the short-range case $\alpha > d$, in which one expects $ \Delta_R(RD_1, RD_2) \asymp R^d$, and the strongly correlated case $\alpha < d$, in which one expects $ \Delta_R(RD_1, RD_2) \asymp  K(R)^{-1}$ (recall that $K(R) \asymp R^{d+1} w(R)$ by Proposition~\ref{p:smprop}).
\end{remark}

 \begin{proof}
  In the proof $c_i > 0$ will denote constants that depend only on $D_1$ and $D_2$. For $R > 0$, we consider functions $h^{a,b} \in H_R$, $0 \le a \le b \le 1$, defined as 
\[ h^{a,b}  = (h_1^{a,b}, h_2^{a,b}) =  ( q_1 \star_1 \varphi^{a,b} , q_2 \star_1 \varphi^{a,b})  \]
where
\[  \varphi^{a,b}(x, t) =  \big(  \id_{R D_1}(x) - \id_{R D_2}(x)  \big) \times \sqrt{w(t)} \id_{t \in ( a R, b R)   } . \]
 In particular we will show that, for sufficiently small $c_1 > 0$ and large $c_2 > 0$, for every $a \le b \le c_1$ we can find a $\lambda \in (0,c_2)$ such that
 \[  \bar{h}^{a,b} = \lambda  R^{-d-1} \Big( \int_{a}^{b} w(Rt) t^d \, dt  \Big)^{-1} h^{a,b} \] 
 is a candidate function for the infimum in \eqref{e:inf}.
 
\smallskip
To this end we first claim that, if $c_1 > 0$ is fixed sufficiently small then, then there exists $c_3 > 0$ such that for every $a \le b \le c_1$
\begin{align}
\label{e:rkhs1}
&  \min \Big\{  h^{a,b}_1(xR)|_{x \in D_1} ,  -h^{a,b}_1(xR)|_{x \in D_2} \Big\}  \\
\nonumber & \qquad \qquad \ge   R^{d+1}   \int_{a}^{b} w(Rt)   \Big( c_3 t^d -  100 \int_{ |y| \ge c_3 } e^{-|y| / (100 t) } \, dy \Big)   dt   ,  
 \end{align}
and
\begin{equation}
\label{e:rkhs2}
\| h^{a,b}_2 \|_\infty  \le R^{d+1} \int_{a}^{b} w(Rt)    \Big( 100 \int_{ |y| \ge c_3 }  e^{-|y| / (100 t) } \, dy \Big) dt    .
\end{equation}
To establish \eqref{e:rkhs1}, we first note the fact that, since $D_i$ are piecewise smooth and Lipschitz, we may fix a $c_4 > 0$ such that, for sufficiently small $t > 0$,
 \begin{equation}
 \label{e:smooth}
  \inf_{y \in D_i} \textrm{Vol} \big( D_i \cap B(y, t) \big) \ge c_4 t^d .  
  \end{equation}
Then if $a \le b \le 25$, by the change of variables $(y,t) \mapsto (Ry,Rt)$, for $x \in D_1$,
\begin{align*}
& h^{a,b}_1(xR) \\
& \quad = R^{d+1} \int_{a}^{b} w(Rt)   \Big(    \int_{D_1} \psi( |x-y| ) Q( |x-y|/t) \, dy  - \int_{D_2} \psi(|x-y|) Q(|x-y|/t)   \, dy  \Big)  dt \\
& \quad \ge R^{d+1}  \int_{a}^{b} w(Rt)  \Big(   \frac{1}{100}  \int_{D_1} \id_{ |x-y| \le t/100 }  \, dy  -100  \int_{D_2} e^{-|x-y|/(100t)}  \, dy  \Big)   dt \\
& \quad \ge R^{d+1}  \int_{a}^{b} w(Rt)   \Big( c_5 t^d       - 100    \int_{ |y| \ge d(D_1,D_2) } e^{-|y| / (100 t) } \, dy  \Big)   dt,
\end{align*}
where the first inequality used that 
\[ \id_{|x| \le 1/4} \le \psi(x) \le  1 \quad \text{and} \quad \frac{1}{100} \id_{|x| \le 1/100} \le Q(x) \le 100 e^{-|x|/100}, \]
 the second inequality used \eqref{e:smooth}. Applying the same argument to $x \in D_2$, the bound \eqref{e:rkhs1} follows. To establish \eqref{e:rkhs2} we similarly have, for $x \in \R^d$,
\begin{align*}
|h^{a,b}_2(xR)| & \le R^{d+1} \int_{a}^{b} w(Rt)   \Big(    \int_{\R^d} (1 - \psi( |x-y| ) ) Q( |x-y|/t) \, dy    \Big)  \, dt \\
& \le 100 R^{d+1} \int_{a}^{b} w(Rt)  \Big(   \int_{\R^d} \id_{ |y| \ge 1/4  }  e^{-|y|/(100t)}  \, dy   \Big)  \, dt 
\end{align*}
where we used that $ 1 - \psi(x) \le \id_{|x| \ge 1/4} $, $Q(x) \le 100 e^{-|x|/100}$.

\smallskip
We next observe that, for sufficiently small $t > 0$,
 \[ 100  \int_{ |y| \ge c_3 } e^{-|y| / (100 t) } \, dy     \le   c_3 t^d/4     . \]
 Hence, comparing \eqref{e:rkhs1} and \eqref{e:rkhs2}, we can fix $c_1$ sufficiently small so that, for every $a \le b \le c_1$,
\[  2 \| h^{a,b}_2 \|_{\infty}   \le   \min \Big\{    \inf_{x \in RD_1} h^{a,b}_1(x) , \inf_{x \in RD_2} -h^{a,b}_1(x)  \Big\}   . \]
 In particular, for every $a \le b \le c_1$ there exist a $\lambda \in (0,4/c_3]$ such that
 \[  \bar{h}^{a,b} = \lambda   R^{-d-1} \Big( \int_{a}^{b} w(Rt) t^d \, dt  \Big)^{-1} h^{a,b} \] 
 is a candidate function for the infimum in \eqref{e:inf}.
 
 \smallskip It remains to bound the RKHS norm of $\bar{h}^{a,b}$ for suitable $a,b$. For general $a > b$,
\begin{align*}
 \label{e:norm}
\|\bar{h}^{a,b}\|^2_{H_R} &  =  \lambda^2 R^{-2d-2} \Big( \int_{a}^{b} w(Rt) t^d \, dt  \Big)^{-2}  \| \varphi^{a,b} \|_{L^2(\R^d \times \R^+)}^2  \\
&  =    \lambda^2  \times \textrm{Vol}( D_1 \cup D_2)   R^{-d-1} \Big( \int_{a}^{b} w(Rt) t^d \, dt  \Big)^{-2} \times  \int_{a}^{b} w(Rt)  \, dt   \\
& \le c_6 R^{-d-1} \Big( \int_{a}^{b} w(Rt) t^d \, dt  \Big)^{-2} \times  \int_{a}^{b} w(Rt)  \, dt    .
\end{align*}
Recall that $w$ is regularly decaying with index $\alpha + d + 1$, so in particular $w(t)$ and $w(t) t^d$ are integrable. Then to establish \eqref{e:rkhsbound1} we choose $a = 0$ and $b = c_1$, and notice that by the change of variables $t \mapsto t/R$, as $R \to \infty$,
\[ \frac{ R^{-2d-1} \int_{0}^{b} w(Rt)  \, dt}{\big( \int_{0}^{b} w(Rt) t^d \, dt  \big)^{2} }  =   \frac{ \int_{0}^{bR} w(t)  \, dt }{   \big( \int_{0}^{bR} w(t) t^d \, dt  \big)^{2} }  \to \frac{\| w (t) \|_{L^1} }{ \|w (t) t^d \|_{L^1}^2} \in (0, \infty) . \]
To establish \eqref{e:rkhsbound2} we instead choose $a = c_1/2$ and $b = c_1$. Then by Potter's bounds \eqref{e:potter} and the dominated convergence theorem, as $R \to \infty$,
 \[    w(R)^{-1} \int_{c_1/2}^{c_1} w(Rt) t^d  \, dt  =   \int_{c_1/2}^{c_1} w(tR) (tR)^d /  ( w(R) R^d )  \, dt   \to   \int_{c_1/2}^{c_1} t^{-\alpha-1} \, dt , \]
 and similarly
 \[ w(R)^{-1} \int_{c_1 /2}^{c_1 } w(t)  \, dt   \to  \int_{c_1/2}^{c_1} t^{-\alpha-d-1} \, dt . \]
Hence
\begin{equation}
\label{e:rkhs3}
\limsup_{R \to \infty}  w(R)^{-1} R^{-d-1} \|\bar{h}^{c_1/2,c_1}\|^2_{H_R}  \le    \frac{ c_6  \int_{c_1/2}^{c_1} t^{-\alpha-d-1} \, dt }{ \big( \int_{c_1/2}^{c_1} t^{-\alpha-1} \, dt  \big)^2} \in (0, \infty) .
\end{equation}
Since the right-hand side of \eqref{e:rkhs3} is a continuous function of $\alpha \ge 0$, the left-hand side is bounded uniformly over $\alpha \in (0,d]$, and \eqref{e:rkhsbound2} follows.
 \end{proof}

\subsection{Weak mixing}
Since $f \in \calS$ is positively-correlated, increasing events are also positively-correlated (Lemma \ref{l:pa}). Our weak mixing result gives a \textit{reverse inequality}, stating roughly that increasing events on disjoint domains are not \textit{too} positively-correlated. 

\begin{proposition}
\label{p:wm}
For every $\eps > 0$ and every pair of disjoint piecewise smooth Lipschitz domains $D_i \subset \R^d$, $i=1,2$, there exists $c_1,c_2 > 0$ such that, for every $f \in \calS_\alpha$, $\alpha > 0$,
\[  \liminf_{R \to \infty} \,  \inf \Big\{  \P[A_1 \cap A_2^c ] : A_i \in I(RD_i), \P[A_1] \ge \eps, \P[A_2^c] \ge \eps \Big\}  \ge  c_1 e^{-c_2 / \alpha^2 }  ,\]
where $I(D)$ denotes the set of increasing events measurable with respect to $f|_D$. 

\end{proposition}

\begin{remark}
In the `short range' case $\alpha > d$, with a slight modification of the proof we could actually establish a much stronger `quasi-independence' result
\begin{equation}
\label{e:qistrong}
 \limsup_{R \to \infty}  \,  \sup \Big\{ \big|  \P[A_1 \cap A_2^c ] - \P[A_1] \P[A_2^c] \big|   :  A_i \in I(RD_i)  \Big\}  = 0 ,
\end{equation}
which is already known under weaker assumptions \cite{mv20,mui23}. On the other hand, we do not expect \eqref{e:qistrong} to hold for arbitrarily small $\alpha < \alpha_0$ (i.e.\ $\alpha_0=3/2$ in the case $d=2$) since this would conflict with the belief that $f \in \calS_\alpha$ is outside the Bernoulli universality class for small~$\alpha$. In the intermediate regime $\alpha \in (\alpha_0,d)$ we expect that \eqref{e:qistrong} holds for percolation-type events, but not for arbitrary increasing events. If $\alpha > 2d$ then one can establish \eqref{e:qistrong} for a wider class of events, not necessarily increasing, see e.g.\ \cite{bmr20}.
\end{remark}

In the proof we make use of the following bound, which is a special case of Lemma \ref{l:ebgen}:
 
\begin{lemma}
\label{l:eb}
For every $R \ge 1$, $h = (h_1,h_2) \in H_R$, and event $A$,
\[ \P[ (f_R  + h_1, g_R + h_2) \in A ] \ge   \P[ (f_R , g_R)   \in A] \exp \Big(   -  \frac{\|h\|_{H_R}^2}{2  \P[ (f_R, g_R)  \in A]}  - 1  \Big)   .\]
\end{lemma}

\begin{proof}[Proof of Proposition \ref{p:wm}]
In the proof $c_i > 0$ will denote constants that may depend only on $\eps$, $D_1$ and $D_2$. By rescaling the domain of $f$, without loss of generality we may assume that $d(D_1,D_2) \ge 1$. We treat the cases $\alpha \le d$ and $\alpha > d$ separately.

\smallskip \noindent \textit{Case $\alpha \le d$.} We first claim that there exist $c_1,c_2 > 0$ such that, for all $s \ge c_1$ and sufficiently large $R$,
\begin{equation}
\label{e:qitail}
 \P \big[ \|g_R \|_{R(D_1 \cup D_2)} \ge s ( \alpha^{-1} w(R)R^{d+1})^{1/2} \big] \le  e^{-(s  -c_1)^2/c_2} .
 \end{equation}
For this, define the rescaled field $\tilde{g}_R(x) =   (\alpha^{-1} w(R) R^{d+1})^{-1/2} g_R(Rx)$, which by Proposition~\ref{p:smprop} satisfies, for every multi-index $k$ with $|k| \le 1$,
\[ \var[\partial^k \tilde{g}_R(0)] = \alpha \Gamma(\alpha) ( \Gamma(\alpha) w(R)R^{d+1-2|k|})^{-1}  \var[\partial^k g_R(0)]    \le c_3  \]
for sufficiently large $R$ (recall that $\alpha \Gamma(\alpha)$ is bounded on $(0, d]$). Then by Kolmogorov's theorem (see \cite[Appendix A.9]{ns16}) and the BTIS inequality (see \cite[Theorem 2.9]{aw09}) 
\[  \P \Big[ \sup_{x \in D_1 \cup D_2} \tilde{g}_R(x)  \ge s \Big]   \le  e^{-(s-c_4)^2/c_3} , \]
which gives \eqref{e:qitail} after rescaling. 

\smallskip
Next recall that, by Proposition \ref{p:rkhsbound},  for sufficiently large $R$
\begin{equation}
\label{e:deltabound}
\Delta_R(RD_1,RD_2) \le  c_5 w(R)^{-1} R^{-d-1},
\end{equation} 
and define the constants $c_6$, $\mu_1 = \mu_1(\alpha)$ and $\mu_2 = \mu_2(\alpha)$ via
\[ e^{-(c_6 - c_1)^2/c_2} =  \eps/2   \ ,  \quad  \mu_1 = \Big( \frac{\eps}{2} e^{-c_6^2 c_5 / ( \eps \alpha) - 1} \Big)^2  \quad \text{and} \quad e^{-(\mu_2 - c_1)^2/c_2} =  \mu_1/2  . \] 
We claim that, provided $R \ge 1$ is taken sufficiently large so that \eqref{e:qitail} and \eqref{e:deltabound} are available,
\[  \inf \Big\{  \P[A_1 \cap A_2^c ] : A_i \in I(RD_i), \P[A_1] \ge \eps, \P[A_2^c] \ge \eps \Big\}   \ge \mu_3  > 0  \]
where  $\mu_3 =\mu_3(\alpha) =  \frac{\mu_1}{2} e^{ - \mu_2^2 c_5 / ( \eps \alpha) - 1}  $. It is simple to check that, as $\alpha \to 0$,
\[ \mu_1 \sim  c_7 e^{-c_8/\alpha}  \ , \quad \mu_2 \sim c_9 / \sqrt{\alpha}  \quad \text{and hence} \quad \mu_3 \sim c_{10} e^{-c_{11}/\alpha^2} , \]
and so this bound completes the proof.

\smallskip
We proceed in two steps. First we show that 
\[  \P[f_R \in A_1 , f_R \in A_2^c ] \ge \mu_1 .\]
Let $t_R = c_6 ( \alpha^{-1} w(R)R^{d+1})^{1/2} $. Using that $A_1$ is increasing and supported on $D_1$,
\[ \P[f_R \in A_1]  \ge \P[ f - t_R \in A_1, \| g_R \|_{\infty, D_1} \le t_R ] .\]
Now let $h = (h_1,h_2) \in H_R$ be such that $h_1|_{D_1} \ge 1$ and $-1/2 \le h_2|_{D_1} \le 1/2$. Then
\begin{align*}
 & \P[ f - t_R \in A_1, \| g_R \|_{\infty, D_1} \le t_R ]  \ge \P[f - t_R h_1 \in A_1,  \| g_R - t_R h_2 \|_{\infty, D_1} \le t_R/2 ] \\
 & \qquad \ge \P[f  \in A_1,  \| g_R  \|_{\infty, D_1} \le t_R/2 ] \exp \Big(    -  \frac{t_R^2 \|h\|_{H_R}^2}{2 \P[f  \in A_1,  \| g_R  \|_{\infty, D_1} \le t_R/2 ] }   - 1 \Big) .
 \end{align*}
where we used Lemma \ref{l:eb} in the second inequality. Using \eqref{e:qitail} and by the definition of $c_6$,
\[ \P[f  \in A_1,  \| g_R  \|_{\infty, D_1} \le t_R/2 ]  \ge  \P[f  \in A_1] -  \P[  \| g_R  \|_{\infty, D_1} \ge t_R/2 ]   \ge \eps - e^{-(c_6-c_1)/c_2 }= \eps/2 .\]
By monotonicity and the definition of $t_R$, we therefore have
\[ \P[f_R \in A_1] \ge \frac{\eps}{2} \exp \Big(    -  \frac{ c_6^2 \alpha^{-1} w(R)R^{d+1} \|h\|_{H_R}^2}{\eps}   - 1 \Big)  .\]
Taking a minimum over all compatible $h$, by \eqref{e:deltabound} and the definition of $\mu_1$,
\begin{align*}
\P[f_R \in A_1]  & \ge   \frac{\eps}{2}  \exp \Big(    -  \frac{ c_6^2 \alpha^{-1} w(R)R^{d+1}  \Delta_R(RD_1, RD_2)}{\eps}   - 1 \Big)  \\
&   \ge   \frac{\eps}{2}  \exp \Big(    -  \frac{c_6^2 c_5}{ \eps \alpha }  - 1 \Big)   = \sqrt{\mu_1} . 
\end{align*}
Since an analogous argument shows that $ \P[f_R \in A^c_2] \ge \sqrt{\mu_1}$, and since $f_R|_{D_1}$ and $f_R|_{D_2}$ are independent, we have
\[  \P[f_R \in A_1 , f_R \in A_2^c ] = \P[f_R \in A_1] \P[ f_R \in A_2^c ]   \ge \mu_1 \]
as claimed.

\smallskip
For the second step, we use the same reasoning to deduce that
\begin{align*}
& \P[f \in A_1, f \in A_2^c] \\
& \qquad \ge \P[f_R \in A_1, f_R \in A_2^c,  \| g_R \|_{\infty, D_1 \cup D_2} \le t_R/2 ] \\
& \qquad \qquad  \times \exp \Big(    -  \frac{t_R^2 \|h\|_{H_R}^2}{2 \P[f_R  \in A_1,  f_R \in A_2^c , \| g_R  \|_{\infty, D_1 \cup D_2} \le t_R/2 ] }   - 1 \Big) 
\end{align*}
for $t_R = \mu_2 (\alpha^{-1} w(R)R^{d+1})^{1/2}$ and every $h = (h_1,h_2) \in H_R$ such that $h_1|_{D_1} \ge 1$, $h_2|_{D_2} \le -1$, and $-1/2 \le h_2|_{D_1 \cup D_2} \le 1/2$. Then since
\begin{align*}
&  \P[f_R  \in A_1,  f_R \in A_2^c , \| g_R  \|_{\infty, D_1 \cup D_2} \le t_R/2 ]  \\
& \qquad \ge  \P[f_R  \in A_1,  f_R \in A_2^c ] - \P[ \| g_R  \|_{\infty, D_1 \cup D_2} \le t_R/2 ] \\
 & \qquad \ge \mu_1 - e^{-(\mu_2 - c_1 )^2/c_2} = \mu_1/2 , 
 \end{align*}
by the definition of $\mu_1$, we have that $ \P[f \in A_1, f \in A_2^c]  $ is bounded below by
\[  \frac{\mu_1}{2} \exp \Big(    -  \frac{ \mu_2^2 \alpha^{-1} w(R)R^{d+1} \Delta_R(RD_1, RD_2)}{\eps}   - 1 \Big)   \ge \frac{\mu_1}{2} \exp \Big(    -   \frac{\mu_2^2 c_5 }{\eps \alpha}  - 1 \Big)  =  \mu_3 , \]
which completes the proof in this case.

\smallskip \noindent \textit{Case $\alpha > d$.} The proof is similar except we replace \eqref{e:qitail} with the bound
\begin{equation}
\label{e:qitail2}
 \P \big[ \|g_R \|_{R(D_1 \cup D_2)} \ge s (  R^{-d} / (\log R) )^{1/2} \big] \le  e^{-(s -c_1)^2/c_2} 
 \end{equation}
 which holds since the rescaled field $\tilde{g}_R(x) =   (R^d  (\log R) )^{1/2} g_R(Rx)$ satisfies, by Proposition~\ref{p:smprop},
\[ \var[\partial^k \tilde{g}_R(0)] =   R^{d + 2|k|} (\log R) \var[\partial^k g_R(0)]    \to 0 \] 
as $R \to \infty$, where we used that $w(R)$ is regularly decaying with index $\alpha + d + 1$ in the final step (and recall $\alpha > d$). The remainder of the proof is the same as the previous case, except that we may choose $\mu_1$ and $\mu_2$ independently of $\alpha$.
\end{proof}

\begin{remark}
We remark that the extra $R^{-2|k|}$ factor in the bound 
\[  \var[\partial^k g_R(0)]  \le c w(R) R^{d+1} R^{-2|k|} \]
in \eqref{e:gbound} is crucial in the case $\alpha \le d$ of the proof of Proposition \ref{p:wm}, since it allows us to show that the supremum of the field $g_R$ on $\Lambda_R$ is of order $w(R) R^{d+1}$. For general $f \in \F_\alpha$, we only have the weaker bound
\[    \var[\partial^k g_R(0)]  \le c w(R) R^{d+1}  .\]
This shows only that the supremum of $g_R$ on $\Lambda_R$ is of order $w(R) R^{d+1} \sqrt{\log R}$, which would be insufficient in the proof.
\end{remark}

\subsection{Applications to nodal connectivity}

In this section we specialise to planar fields $f \in \calS$, and show how Proposition \ref{p:wm} is used to deduce Theorems \ref{t:main2} and \ref{t:main3} and Corollary~\ref{c:main3}. 

\smallskip
We begin with Theorem \ref{t:main3}. Apart from Proposition \ref{p:wm}, the main input into the proof is a uniform box-crossing property for the excursion sets $\{f \ge 0\}$, recently established in \cite{kt20}. For $a,b, > 0$, recall the box-crossing event $\nodcross(a,b)$, and analogously let $\cross(a,b)$ denote the event that $\{f \ge 0\} \cap ([0,a] \times [0,b])$ contains a path that crosses $[0,a] \times [0,b]$ from left-to-right.

\begin{theorem}[Uniform box-crossing property for the excursion sets; \cite{kt20}]
\label{t:rsw}
For every $\rho > 0$ there is a $c > 0$ such that, for all $f \in \calS$ and $R > 0$,
\[ c \le  \P[ \cross(R, \rho R) ]  \le    \P[ \cross(R, \rho R) ]  \le 1-c  .\]
\end{theorem}

\begin{remark}
Note that the constant $c$ can be chosen independently of $f \in \calS$, and in particular does not depend on $\alpha$. Indeed, in \cite{kt20} the box-crossing property is shown to hold uniformly over all random stationary closed subsets of the plane satisfying:
\begin{itemize}
\item Symmetry under axes reflection and rotation by $\pi/2$.
\item Squares are crossed with probability $1/2$ (see item (5) of Corollary \ref{c:cont}).
\item Positive associations (see Lemma \ref{l:pa}).
\end{itemize}
\end{remark}

\begin{proof}[Proof of Theorem \ref{t:main3}]
By isotropy, and since nodal lines cannot intersect, it is sufficient to prove the lower bound. Observe that, by the continuity of $f$,
\begin{align*}
&  \{ f \in \cross(R, \rho  R / 4) \}   \cap \{  -f  \in (0, 3\rho R/4)  + \cross(R, \rho R/4)   \}  \\
 & \qquad \qquad \implies \{ f \in \nodcross( R, \rho R ) \} . 
 \end{align*}
By stationarity, self-duality, and Theorem \ref{t:rsw}, 
\[ \P[  f \in \cross(R, \rho  R / 4) ] = \P[     -f  \in (0, 3\rho R/4)  + \cross(R, \rho R/4)  ]   \ge c_1 \]
for some $c_1 > 0$ depending only on $\rho$. The conclusion then follows from Proposition \ref{p:wm}.
\end{proof}

\begin{proof}[Proof of Corollary \ref{c:main3}]
By the union bound and Theorem \ref{t:main3}, for every $R \ge 1$,
\[ \P[ \nodarm(R)] \ge c_1R^{-1} \P[\nodcross(R, R)]  \ge c_2 R^{-1}\]
for constants $c_1,c_2 > 0$. Moreover, observe that 
\[ \{f \in \nodarm(R) \} \implies \{f \in \arm(R) \} \cap \{ -f \in \arm(R) \} , \]
and so, by positive association (Lemma \ref{l:pa}) and self-duality,
\[  \P[ f \in \nodarm(R) ] \le \P[ \{f \in \arm(R) \} \cap \{ -f \in \arm(R) \} ] \le \P[\arm(R) ]^2 . \]
Hence $\P[\arm(R) ] \ge   \P[ f \in \nodarm(R) ]^{1/2} \ge c_2^{1/2}R^{-1/2}$.
\end{proof}

To finish we use Corollary \ref{c:main3} (together with Theorem \ref{t:rsw}) to establish the improved density bound in Theorem \ref{t:main2}. The proof is very similar to Theorem \ref{t:main1}, so we just explain the necessary modification.

\begin{proof}[Proof of Theorem \ref{t:main2}]
The crux of the modification is to use the bound
\begin{equation}
\label{e:main21}
  \P_\ell[  \Lambda_s \leftrightarrow  \partial \Lambda_r   ]  \le c_1 s^{1/2}  \P_\ell[  \Lambda_{1/2} \leftrightarrow  \partial \Lambda_r   ]  ,
  \end{equation}
valid for all $\ell \ge \ell_c$, $1 \le s \le r/2$, and a constant $c_1 > 0$, in place of the union bound
\[  \P_\ell[  \Lambda_s \leftrightarrow  \partial \Lambda_r   ]  \le c_1 s^{d-1}  \P_\ell[  \Lambda_{1/2} \leftrightarrow  \partial \Lambda_r   ]  \]
in the proof of Claim \ref{c:algo} (and in particular \eqref{e:algo1}). So let us begin by proving \eqref{e:main21}.

\smallskip
Let $\ann(s)$ be the event that $\{f \ge 0\} \cap ( \Lambda_{2s} \setminus \Lambda_s)$ contains a circuit in the annulus $ \Lambda_{2s} \setminus \Lambda_s$. By Theorem \ref{t:rsw}, positive associations, and standard gluing arguments, there is a $c_2 > 0$ such that $\P_\ell[\ann(s)] \ge c_2$ for all $\ell \ge \ell_c$ and $s > 0$. Now observe that, for $1 \le s \le r/2$,
\[ \{ \Lambda_{1/2} \leftrightarrow \partial \Lambda_{2s} \} \cap \{\Lambda_s \leftrightarrow \partial \Lambda_r \} \cap \ann(s)  \implies \{\Lambda_{1/2} \leftrightarrow \Lambda_r \} .\]
Hence, by positive associations, for every $\ell \ge \ell_c$,
\[ \P_\ell[ \Lambda_{1/2} \leftrightarrow \partial \Lambda_r] \ge  c_2 \P_\ell[ \Lambda_{1/2} \leftrightarrow \partial \Lambda_{2s} ] \P[\Lambda_s \leftrightarrow  \partial \Lambda_r] .\]
Finally, by Corollary \ref{c:main3},
\[ \P_\ell[ \Lambda_{1/2} \leftrightarrow \partial \Lambda_{2s} ] \ge c_3 s^{-1/2} \]
for some $c_3 > 0$, and combining gives \eqref{e:main21}.

\smallskip
Now recall the notation from the proof of Theorem \ref{t:main1}. Using \eqref{e:main21} in place of the union bound in Claim \ref{c:algo} gives the following modification of the claim:

\begin{claim}
\label{c:algo2}
For every $m \in \N$, $\ell \ge \ell_c$, and $R \ge 1$, there exists a class-$m$ randomised algorithm determining $A$ such that, for all $m' \le m-1$,
\[ \sup_{i \in \Z^d} \Rev_{i, m'} \le  \frac{c_6 (2^{m'})^{1/2}}{R} \sum_{r = 1}^{\lceil R \rceil} \theta_r(\ell)  \]
where $c_6 > 0$ is a constant independent of $R$ and $\ell$.
\end{claim}

In turn, replacing Claim \ref{c:algo} with Claim \ref{c:algo2} in the proof of \eqref{e:mens4} establishes the following modification of \eqref{e:mens4}: for every $\eps \in (0, \alpha)$ and $\ell_0 \in \R $ there exists a $c > 0$ such that, for all $R \ge 1$ and $ \ell \in [\ell_c, \ell_0) $,
 \[   \frac{d^+}{d\ell} \theta_R(\ell) \ge \frac{c \, \theta_R(\ell)}{ \big( \frac{1}{R} \sum_{r \ge 1}^{\lceil R \rceil}\theta_r(\ell) \big)^{\gamma}  }  ,\]
 for $\gamma = \min\{ 1 , \alpha - \eps \id_{\alpha = 1} \}$. Given this, the result follows from Lemma \ref{l:mens} applied to the function $ \theta_R(\ell)$ restricted to $R \in \N$ and $\ell \in [\ell_c,\ell_0]$.
\end{proof}

\begin{remark}
\label{r:onearm}
The proof of Theorem \ref{t:main2} shows that if there is an exponent $\eta_1 > 0$ such that
\[ \P[\arm(R)]  \ge c R^{-\eta_1},  \ R \ge 1, \]
then \eqref{e:main2} holds with the constant $\gamma \in (0,\alpha/(2\eta_1)) \cap (0,1]$. In particular, the mean-field bound $\beta \le 1$ holds if $\eta_1 \le \alpha/2$. Corollary \ref{c:main3} shows that one may take $\eta_1 = 1/2$.
\end{remark}

\smallskip

\appendix

\section{}

We collect properties of Gaussian fields $f = q \star_1 W|_D$, where $q \in L^2(\R^d \times \R^+)$ such that $q(x,t) = q(-x,t)$, $W$ is the white noise on $\R^d \times \R$, and $D \subset \R^d \times \R^+$ is a Borel set.

\smallskip We let $K(x,y) = \E[f(x) f(y)] = \langle q(\cdot- x), q(\cdot-y) \rangle_{L^2(D)}$ denote the covariance kernel of $f$. If $D = \R^d \times T$ for some $T \subset \R^+$, then $f$ is stationary, and we write $K(x) = \E[f(0)f(x)] = (q|_D \star_1 q|_D)(x) $. In that case we also denote by $\rho = \mathcal{F}[K]$ the spectral measure of $f$, where $\mathcal{F}[\cdot]$ denotes the Fourier transform on $\R^d$.

\subsection{Smoothness and non-degeneracy}
\label{s:cont}
We first establish consequences of Assumption~\ref{a:main}:

\begin{lemma}
\label{l:cont}
$\,$
\begin{itemize}
\item Suppose $q$ satisfies the smoothness in Assumption \ref{a:main}. Then $f$ is a.s.\ $C^2$-smooth, and for every multi-index $k$ with $|k| \le 2$,
\begin{equation}
\label{e:equal}
\partial^k f = ( \partial^{k,0} q ) \star_1 W|_D .  
\end{equation}
\item Suppose $q$ satisfies the non-degeneracy in Assumption \ref{a:main} and $D = \R^d \times \R^+$. Then the support of the spectral measure $\rho = \mathcal{F}[K]$ contains an open set.
\end{itemize}
\end{lemma}
\begin{proof}
For the first item, fix $\psi: \R^+ \to [0, 1]$ a smooth function such that $\psi(x) = 1$ for $x \le 1/4$ and $\psi(x) = 0$ for $x \ge 1/2$, and for every $n \in \N$, define the Gaussian field $f_n(x) = q_n \star_1 W|_D$, where $q_n(x,t) = \psi(|x|/n)  \id_{t \le n} \, q(x,t)$. Let $K_n(x,y)= \E[f_n(x) f_n(y)]$ be the covariance kernel of $f_n$. Consider a multi-index $k$ with $|k| \le 3$ and $D' \subset \R^d$ a compact set. Since $\partial^{k,0} q$ is locally bounded and $\psi$ is smooth, $ \partial^{k,0}  q_n$ is bounded and has compact support. Hence, using the dominated convergence theorem to exchange derivative and convolution, we see that $K_n \in C^{3,3}(D')$, and the conclusion for $f_n$ in place of $f$ follows (see, e.g., \cite[Proposition 3.3]{mv20} for details). On the other hand, since $\partial^{k,0} q \in L^2(\R^d \times \R^+)$ and $\psi$ is smooth, $(f_n)_{n \in \N}$ is a Cauchy sequence in the seminorm
\[  \| f_n \| = \sup_{  x \in D' }  \sup_{|k_1|, |k_2| \le 3}  \partial^{k_1}_x \partial^{k_2}_y \E[  f_n(x)  f_n(y) ] . \] 
In particular, by Kolmogorov's theorem (see \cite[Appendix A.9]{ns16}), $f_n \to f$ in $C^2(D')$ a.s., and so $f$ is $C^2$-smooth. Moreover, for every $|k| \le 2$ and $x \in D'$, both $\partial^k f(x)$ and $((\partial^{k,0} q) \star_1 W|_D) (x)$ are $L^2$-limits of $\partial^k f_n(x) = ((\partial^{k,0} q_n) \star_1 W|_D) (x)$. Hence $\E[ ( \partial^k f(x) - ((\partial^{k,0} q) \star_1 W|_D) (x) )^2] = 0$, and so by considering a continuous modification of $(\partial^{k,0} q) \star_1 W|_D$ we have \eqref{e:equal}.
 
\smallskip 
For the second item, recall that there exists $t \ge 0$ such that $q(\cdot,t)$ is non-zero and $q(\cdot,t') \to q(\cdot, t)$ in $L^1(\R^d)$ as $t' \to t$. In particular, this implies that $ \mathcal{F}[ q(\cdot,t) ](x)$ is continuous and not identically zero on $\R^d \times (t \pm \delta)$ for sufficiently small $\delta$. So let $T = t + (-\delta,\delta)$, and let $f_T = q|_{\R^d \times T} \star_1 W$. By Fubini's theorem, the spectral measure of $f_T$ is
\[ \rho_T(x) = \int_{T} \mathcal{F}[ q(\cdot,t) ]^2 \, dt ,\]
and hence, by continuity, the support of $\rho_T$ contains an open set. By linearity of the Fourier transform, $\rho \ge \rho_T$, which concludes the proof.
 \end{proof}

\begin{remark}
\label{r:smooth}
The same proof shows that if $q(\cdot,t)$ is smooth for every $t \ge 0$, and for every multi-index $k$, $\partial^{k,0} q \in L^2(\R^d \times \R^+)$ and $\partial^{k,0} q$ is locally bounded on $\R^d \times \R^+$, then $f = q \star_1 W$ is a.s.\ smooth.
\end{remark}

The following are standard corollaries of the conclusion of Lemma \ref{l:cont}:
\begin{corollary}
\label{c:cont}
Suppose $q$ satisfies Assumption \ref{a:main} and $D = \R^d \times \R^+$. Then:
\begin{enumerate}
\item  $(f(x), \nabla f(x))$ is non-degenerate for every $x \in \R^d$.
\item The critical points of $f$, as well as its restriction to a smooth hypersurface, are a.s.\ locally finite and have distinct critical levels that are not equal to a fixed level $\ell \in \R$.
\item For every $\ell \in \R$, $\{f = \ell\}$ is a.s.\ a collection of $C^2$-smooth simple hypersurfaces, which a.s.\ intersect a fixed smooth hypersurface transversally.
\item For all $0 < r \le R$, $\{f \in \Lambda_r \leftrightarrow \partial \Lambda_R\}$ is a continuity event.
\item If $d=2$, then for all $R > 0$, $\P[A_R] + \P[B_R] = 1$, where $A_R$ (resp.\ $B_R$) is the event that  $[0,R]^2 \cap \{f \ge 0\}$ contains a path that crosses $[0, R]^2$ from left-to-right (resp.\ top-to-bottom). In particular, if $q$ is invariant under axes reflection and rotation by $\pi/2$, then $\P[A_R] = \P[B_R] = 1/2$.
\end{enumerate}
\end{corollary}
\begin{proof}
The first item is proven in \cite[Lemma A.2]{bmm20} respectively, the second item is a consequence of Bulinskaya's lemma \cite[Lemma 11.2.10]{at07}, and the remaining items are standard consequences of the second item and the fact that $f$ is $C^2$-smooth.
\end{proof}

\begin{corollary}
\label{c:cont2}
Suppose $q$ satisfies only the non-degeneracy in Assumption \ref{a:main} and $D = \R^d \times \R^+$. Then for all $0 < r \le R$, $\{f|_{\Z^d} \in \Lambda_r \leftrightarrow \partial \Lambda_R\}$ is a continuity event.
\end{corollary}
\begin{proof}
Since the support of $\rho$ contains an open set, $K$ is strictly positive definite \cite[Theorem 6.8]{wen05}, and the result follows immediately.
\end{proof}

\subsection{Orthonormal expansion}
We next present a standard orthonormal expansion of $f$:

\begin{lemma}
\label{l:c0conv}
Let $(\varphi_j)_{j \in \N}$ be an orthonormal basis of $L^2(D)$, and let $Z = (Z_j)_{j \in \N}$ be a sequence of i.i.d.\ standard Gaussians. Then
\[ f_n := \sum_{j \ge 1}^n Z_j (q \star_1 \varphi_j) \Rightarrow f  \]
in the sense of finite-dimensional distributions. If additionally $f$ is continuous, the convergence holds in law in the $C^0(\R^d)$-on-compacts topology.

\end{lemma}
\begin{proof}
For each $x \in \mathbb{R}^d$, $f_n(x) \Rightarrow f(x)$ in law since they are centred Gaussian random variables and
\[  \mathbb{E} \Big[ \Big(\sum_{j \ge 1}^n Z_j (q \star \varphi_j)(x)  \Big)^2 \Big] = \sum^n_{j \ge 1} \Big(\int_S q(x-s) \varphi_j(s) \, ds \Big)^2  \to   \int_S q(x-s)^2  \, ds    = \mathbb{E}[f(x)^2]  \]
 by Parseval's identity. Applying the same argument to arbitrary linear combinations, this proves the first statement by the Cram\'{e}r-Wald theorem.
 
For the second, note that, for every $n \ge 1$, $q \star_1 \varphi_n$ is continuous as a convolution of $L^2$ functions, and hence $f_n$ is a continuous Gaussian field. Since $f$ is continuous by assumption, the claim follows from the It\^{o}-Nisio theorem (see \cite[Lemma A.2]{dm21}). 
\end{proof}

\subsection{The reproducing kernel Hilbert space}
\label{s:rkhs}

The \textit{reproducing kernel Hilbert space} (RKHS) associated to $f$ is the space
\[   H = (h^\varphi)_{\varphi \in L^2(D)} = (  q \star_1 \varphi  \big)_{ \varphi \in L^2(D)}   \]
equipped with the inner product 
\[  \langle h^{\varphi_1}, h^{\varphi_2} \rangle_{H} =     \langle  q \star_1 \varphi^1 , q \star_1 \varphi^2 \rangle_{H} = \langle \varphi_1, \varphi_2 \rangle_{L^2(D)} .\]
This is the unique Hilbert space of functions on $\R^d$ satisfying the \textit{reproducing property}
\[ \langle  K(\cdot,x), h^\varphi \rangle_{H} =  \langle ((q \star_1 q)(\cdot-x)),  q \star_1 \varphi  \rangle_{H}  =  \langle  q(\cdot-x), \varphi \rangle_{L^2(D)} =  (q \star_1 \varphi)(x) = h^\varphi(x) . \]

Now consider a pair of continuous Gaussian fields $(f_1,f_2)  = ( q_1 \star_1 W|_D, q_2 \star_1 W|_D)$, where $q_1,q_2 \in L^2(\R^d \times \R^+)$. Then the RKHS of $f_1$ and $f_2$ can be naturally paired together to form the Hilbert space
\[   H^2  = \big( ( h^{\varphi}, h^{\varphi} ) \big)_{\varphi \in L^2(D)} =  \big(  ( q_1 \star_1 \varphi , q_2 \star_1 \varphi) \big)_{\varphi \in L^2(D)}   \] 
equipped with the inner product 
\[  \big\langle (h^{\varphi_1}_1,h^{\varphi_1}_2), (h^{\varphi_2}_1, h^{\varphi_2}_2) \big\rangle_{H^2} =   \langle h_1^{\varphi_1}, h_1^{\varphi_2} \rangle_{H}  =  \langle h_2^{\varphi_1}, h_2^{\varphi_2} \rangle_{H}  = \langle \varphi_1, \varphi_2 \rangle_{L^2(D)} .\]

\smallskip
The main property of the space $H^2$ that we need is the following:

\begin{lemma}
\label{l:ebgen}
For every $h = (h_1^\varphi,h^\varphi_2) \in H^2$ and event $A$,
\[ \P[ (f_1  + h^\varphi_1, f_2 + h^\varphi_2) \in A ] \ge   \P[ (f_1, f_2)   \in A] \exp \Big(   -  \frac{\|h^\varphi\|_{H^2}^2}{2  \P[ (f_1, f_2)  \in A]}  - 1  \Big)   .\]
\end{lemma}

\begin{proof}
We first claim that
\begin{equation}
\label{e:cmt}
 d_{KL} \big(    (f_1,f_2)  || (f_1+h^\varphi_1, f_2+ h^\varphi_2 ) \big)  \le \frac{\|h\|^2_{H^2}}{2} , 
 \end{equation}
where $d_{KL}(X ||Y)$ denotes the relative entropy from $Y$ to $X$. To see this, observe that Lemma~\ref{l:c0conv} gives the decomposition
\[ (f_1,f_2) \stackrel{d}{=} \Big( \sum_{j \ge 1} Z_j  (q_1 \star_1 \varphi_j) ,   \sum_{j \ge 1} Z_j (q_2 \star_1 \varphi_j) \Big)  ,\]
where $\varphi_j$ is an orthonormal basis of $L^2(D)$, $Z_j$ are i.i.d.\ standard Gaussians, and we are free to choose $\varphi_1 = \varphi / \| \varphi \|_{L^2(D)}$. Similarly, 
\begin{align*}
& (f_1+h^\varphi_1, f_2+ h^\varphi_2 ) \\
& \qquad = \Big( \sum_{j \ge 1} (Z_j +   \| \varphi \|_{L^2(D)} \id_{j=1}) (q_1 \star_1 \varphi_j),   \sum_{j \ge 1} (Z_j +   \| \varphi \|_{L^2(D)} \id_{j=1} )(q_2 \star_1 \varphi_j) \Big) . 
\end{align*}
Since relative entropy contracts under measurable mappings,
\[  d_{KL} \big(    (f_1,f_2)  || (f_1+h^\varphi_1, f_2+ h^\varphi_2 ) \big)   \le  d_{KL} \big(  Z_1 || Z_1 +  \| \varphi \|_{L^2(D)}  \big)    =  \frac{    \| \varphi \|_{L^2(D)}^2  }{2}  = \frac{ \|h\|_{H^2}^2 }{2} ,\]
which gives \eqref{e:cmt}. The result then follows from the general fact that
\[ \P[ Y \in A ] \ge   \P[ X  \in A] \exp \Big(   -  \frac{  d_{KL}(X || Y)  }{ \P[X \in A] }  - 1  \Big)   \]
for arbitrary random variables $X$ and $Y$ taking values in a common measurable space (see \cite[Lemma 2.5]{dm21b}).
\end{proof}

\begin{remark}
In fact the Cameron-Martin theorem shows that \eqref{e:cmt} is actually an equality, but since the setting is slightly non-standard we prefer not to evoke this theorem.
\end{remark}

\subsection{Positive associations}
To finish we state the well-known property of positive association for positively-correlated Gaussian fields:

\begin{lemma}[Positive associations]
\label{l:pa}
Suppose $K \ge 0$. Then for increasing events $A$ and $B$,
\[ \P[A \cap B] \ge \P[A] \P[B] .\]
\end{lemma}
\begin{proof}
In the finite-dimensional case this is proven in \cite{pit82}, and see \cite[Lemma 2.1]{drrv21} for approximation arguments that extend the result to continuous fields.
\end{proof}

\bigskip

\bibliographystyle{halpha-abbrv}
\bibliography{sharp}

\end{document}